\documentclass[11pt,reqno]{amsart}
\usepackage{amsaddr}
\usepackage{float}
\usepackage{caption}
 \usepackage{amssymb,amsfonts,amscd,amsbsy, color, amsmath}
\usepackage[mathscr]{eucal}
\usepackage{url}
\usepackage{graphicx, csquotes}
\usepackage{mathtools}
\usepackage{todonotes}
\usepackage{tabularx}
\usepackage{hyperref}
\usepackage{float, color}
\usepackage{placeins}
\usepackage{tikz}
\usepackage[labelsep=period]{caption}
\newtheorem{theorem}{Theorem}[section]
\newtheorem{lemma}[theorem]{Lemma}
\newtheorem{proposition}[theorem]{Proposition}

\theoremstyle{definition}

\numberwithin{equation}{section}

\newcommand{\Z}{\mathbb{Z}}

\usepackage{makecell}

\usepackage{relsize}
\usepackage{stmaryrd}
\usepackage{array}

\usepackage[nohyperlinks]{acronym}
\renewcommand{\url}[1]{#1}

\makeatletter
\newcommand*{\rom}[1]{\expandafter\@slowromancap\romannumeral #1@}
\makeatother

\makeatletter
\def\imod#1{\allowbreak\mkern5mu({\operator@font mod}\,\,#1)}
\makeatother
\allowdisplaybreaks

\makeatletter
\newcommand*\bigcdot{\mathpalette\bigcdot@{.5}}
\newcommand*\bigcdot@[2]{\mathbin{\vcenter{\hbox{\scalebox{#2}{$\m@th#1\bullet$}}}}}
\makeatother
\newcommand{\Mod}[1]{\ (\mathrm{mod}\ #1)}

\makeatletter
\@namedef{subjclassname@2020}{%
  \textup{2020} Mathematics Subject Classification}
\makeatother

\begin{document}

\title{Congruence properties modulo prime powers for a class of partition functions}

\author{Matthew Boylan \and Swati}
\address{Department of Mathematics, University of South Carolina, Columbia, SC, 29208, USA}
\email{boylan@math.sc.edu, s10@email.sc.edu}

\subjclass[2020]{11F33, 11F11, 11P83.}
\keywords{Modular forms, congruences, partitions}

\date{}
\begin{abstract}
Let $p$ be prime, and let $p_{[1,p]}(n)$ denote the function whose generating function is $\prod (1-q^n)^{-1}(1 - q^{pn})^{-1}$.  This function and its generalizations $p_{[c^{\ell}, d^m]}(n)$ are the subject of study in several recent papers.  Let $\ell\geq 5$, let $j\geq 1$, and let $p \in \{2, 3, 5\}$.  In this paper, we prove that the generating function for $p_{[1, p]}(n)$ in the progression $\beta_{p, \ell, j}$ modulo $\ell^j$ with $24\beta_{p, \ell, j} \equiv p + 1\Mod{\ell^j}$ lies in a Hecke-invariant subspace of type $\{\eta(Dz)\eta(Dpz)F(Dz) : F(z) \in M_{s}(\Gamma_0(p), \chi)\}$ for suitable $D\geq 1$, $s\geq 0$, and character~$\chi$.  When $p\in \{2, 3, 5\}$, we use the Hecke-invariance of these subspaces proved in \cite{warnockthesis} to prove, for distinct primes $\ell$ and $m\geq 5$ and $j\geq 1$, congruences of the form 
\[
p_{[1, p]}\left(\frac{\ell^jm^k n + 1}{D}\right)\equiv 0\Mod{\ell^j}
\]
for all $n\geq 1$ with $m\nmid n$, where $k$ is explicitly computable and depends on the forms in the invariant subspace.   Our proofs require adapting and extending analogous level one results on $p(n)$ in \cite{ahlgrenboylan} and \cite{yangpartition} to level $p$.
\end{abstract}

\maketitle
\section{Introduction and Statement of Results}

Let $c$ and $d$ be positive integers, and $\ell$ and $m$ be arbitrary integers. The arithmetic function $p_{[c^{\ell}, d^m]}(n)$ defined by 
\begin{equation}
\label{genfun}
\sum_{n = 0}^{\infty}p_{[c^{\ell}, d^m]}(n)q^n = \prod_{j = 1}^{\infty}\frac{1}{(1 - q^{cj})^{\ell}(1 - q^{dj})^m}
\end{equation}
is the subject of study in several recent papers, including \cite{atmanicongruences}, \cite{chandistribution}, \cite{chancubicpartition}, \cite{chancooper}, \cite{chantoh}, \cite{kimpartition}, \cite{mauth2023exact}, \cite{mercapartitions}, \cite{mestridge}, \cite{sinickcongruences}, \cite{wang}, and \cite{yaopartition}. The papers \cite{chandistribution}, \cite{chancubicpartition}, \cite{kimpartition}, \cite{mauth2023exact}, \cite{mercapartitions} and \cite{yaopartition} focus on arithmetic properties of $p_{[1, 2]}(n)$; \cite{chancooper} concerns $p_{[1^2, 3^2]}(n)$; \cite{chantoh} studies $p_{[1^{\ell}, c^m]}$ for $c\in \{2, 3, 5, 7, 11\}$ and certain $\ell$, $m\leq 4$; \cite{mestridge} and \cite{wang} prove Ramanujan-type congruences for $p_{[1^{\ell}, 11^{m}]}(n)$ modulo powers of $11$; and \cite{atmanicongruences} proves Ramanujan-type congruences for $p_{[1^{\ell}, d^{km}]}(n)$ modulo powers of $d$ where $d \in \{5, 7, 11, 13, 17 \}$ and $k \geq 1$ is an integer.  In \cite{sinickcongruences}, Sinick classified all primes $\ell$ and integers $N$ and $a$ with $N\geq 2$ such that $p_{[1, N]}(\ell n + a)\equiv 0\Mod{\ell}$ for all $n \in \mathbb{Z}$. 

Let $p$ be prime.  In this paper, we study the specialization $(c, d, \ell, m) = (1, p, 1, 1)$.  
Some of the interest in the functions $p_{[c^{\ell}, d^m]}(n)$ comes from their combinatorial interpretation. For~example, our specialization, $p_{[1, p]}(n)$, counts the number of partitions of $n$ using two colors where one of the colors, say the color red, only has parts that are multiples of $p$.  Interest in $p_{[c^{\ell}, d^m]}(n)$ also comes from the fact that its infinite product generating function is a weakly holomorphic eta-quotient with weight $-1$ and level $p$ with respect to the multiplier system associated to the the Dedekind eta-function, $\eta(z)$.  The eta-function is a function of the complex upper half-plane defined by
\begin{equation}
\label{eta1}
\eta(z) = q^{\frac{1}{24}}\prod_{n = 1}^{\infty}(1 - q^n), \ \ q = e^{2\pi iz}.
\end{equation}
For details on eta-quotients and modular forms, see Section 2.  Since $\eta(z)$ is a modular form, one can use the theory of modular forms to study $p_{[c^{\ell}, d^m]}(n)$ as the authors cited above do, and as we do here.  Our purpose in this paper is to extend and adapt two specific types of theorems on $p(n)$, the ordinary partition function, to the functions $p_{[1,p]}(n)$.  The generating functions for $p(n)$ and $p_{[1,p]}(n)$ reside in spaces that differ in important ways; in the context of the eta-multiplier, the generating function for $p(n)$ has half-integral weight and level $1$, while the generating function for $p_{[1, p]}(n)$ has integral weight and level $p$.  With more work, it is likely that our results are extendable and adaptable to the general functions $p_{[c^{\ell}, d^m]}(n)$ and beyond; our study of the infinite family $p_{[1, p]}(n)$ illustrates what is possible.    

We now give the two theorems on $p(n)$ whose analogues we aim to prove for $p_{[1, p]}(n)$.  Let $\ell \geq 5$ be prime, let $j\geq 1$ be an integer, and let $0\leq \beta_{\ell, j}\leq \ell^j - 1$ be the unique integer satisfying $24\beta_{\ell, j} \equiv 1 \Mod{\ell^j}$.  We also require integers 

\begin{gather*}
	\displaystyle{	r_{\ell,j} = 
		\begin{dcases}
			24 \left( 1 + \left\lfloor \frac{\ell}{24}  \right\rfloor   \right) - \ell \quad &\text{if $j$ is odd},\quad \\
			 23 \quad &\text{if $j$ is even}; \quad
		\end{dcases} 
  }
\end{gather*}
 and 
\begin{gather*}
	\displaystyle{	k_{\ell,j} = 
		\begin{dcases}
			\frac{(\ell^{j - 1} + 1)(\ell - 1)}{2} - 12 \left( 1 + \left\lfloor \frac{\ell}{24} \right\rfloor  \right) \quad & \text{ if $j$ is odd}, \\
		\ell^{j - 1} (\ell - 1) - 12 \quad & \text{ if $j$ is even.}
		\end{dcases}
  }
\end{gather*}
The first theorem identifies the generating function for $p(\ell^jn + \beta_{\ell,j})\Mod{\ell^j}$ as a modular form modulo $\ell^j$.  When $N$ and $k$ are integers with $N\geq 1$, and $\psi$ is a Dirichlet character modulo $N$, we denote the space of holomorphic modular forms on $\Gamma_0(N)$ with weight $k$ and nebentypus $\psi$ by $M_{k}(\Gamma_0(N),\psi)$.  
 
\begin{theorem}[Theorem 3 of \cite{ahlgrenboylan}] \label{ahlgren_boylan}
\label{AB}
Assume the notation above.  Let $\ell\geq 5$ be prime, and let $j\geq 1$ be an integer.  Then there exists a modular form $F_{\ell,j}(z) \in M_{k_{\ell,j}}(\Gamma_{0}(1)) \cap \mathbb{Z}[[q]]$ such that
\begin{gather*}
\sum_{n = 0}^{\infty}p(\ell^j n + \beta_{\ell, j})q^{n+\frac{r_{\ell, j}}{24}} \equiv \eta(z)^{r_{\ell,j}} F_{\ell,j}(z) \Mod{\ell^{j}}. 
 \end{gather*}
\end{theorem}

The second theorem concerns explicit congruences for $p(n)$ modulo prime powers. The Ramanujan congruences \cite{RamanujanAiyangarSrinivasa1927} assert, for all $n$, that 
\[
p(\ell n + a)\equiv 0 \Mod{\ell}
\]
when $(\ell, a)\in \{(5, 4), (7, 5), (11, 6)\}$.  Ramanujan's work inspired substantial subsequent study of congruences for $p(n)$.  A major breakthrough occurred when Ono \cite{onopartitions} proved, for all primes $\ell\geq 5$ and all integers $k\geq 1$, that a positive proportion of primes $m$ have the property that $p\left(\frac{\ell^km^3n + 1}{24}\right)\equiv 0\Mod{\ell}$ for all $n$ with $m \nmid n$.  Later, Y. Yang \cite{yangpartition} proved the following theorem which one can use to find explicit congruences for $p(n)$ of the type in Ono's theorem.  

\begin{theorem}[Theorem 6.7 of \cite{yangpartition}] \label{Yang}
Let $\ell$ and $m$ be distinct primes with $\ell \geq 13$ and $m\geq 5$, and let $j\geq 1$ be an integer.  Then there exist explicitly computable integers $K$ and $M\geq 1$ such that the following congruences hold.
\begin{enumerate}
\item 
For all positive integers $u$ and $n$ with $m \nmid n$, we have
\begin{equation*} \label{Yang:result}
    p \left( \frac{\ell^{j} m^{2uK - 1} n + 1}{24}    \right) \equiv 0 \Mod{\ell^{j}}. 
\end{equation*}
\item 
For all nonnegative integers $r$ and $n$, we have 
\begin{equation*}
p\left(\frac{\ell^jm^rn +1}{24}\right) \equiv p\left(\frac{\ell^jm^{M + r}n + 1}{24}\right) \Mod{\ell^j}.
\end{equation*}
\end{enumerate}
\end{theorem}
\noindent
We remark that \cite{chensun} gives analogous results to Theorem \ref{ahlgren_boylan} and \ref{Yang} for the $r$-colored partition function, denoted by $p_{r}(n)$, when $r\geq 1$ is odd.  As the generating function for $p_r(n)$ is the $r$th power of the generating function for $p(n)$, it also has half-integral weight and level~$1$ with respect to the eta-multiplier; therefore, the analogous results for these functions follow similarly.    

\medskip

We now state our theorems on $p_{[1,p]}(n)$.  Since $p_{[1, p]}(n)$ depends on the prime $p$, the notation required to state our theorems in full generality is more complicated.  We let $\ell \neq p$ be prime with $\ell\geq 5$, and we let $j$ be a positive integer. We define $0\leq \beta_{p,\ell,j}\leq \ell^j - 1$ to be the unique integer such that 
	\begin{gather} \label{B1}
		24 \beta_{p,\ell,j} \equiv p + 1 \Mod{\ell^{j}}.
	\end{gather}
	We also require
	\begin{gather} \label{B2}
		\delta_{p,\ell,j} = \frac{(\ell^{2j} - 1) (p + 1)}{24}.
	\end{gather}
Since $\ell\geq 5$ is prime, definitions \eqref{B1} and \eqref{B2} imply that $\delta_{p,\ell,j} \equiv -\beta_{p,\ell, j} \Mod{\ell^j}$, and hence, that $\frac{\beta_{p, \ell, j} + \delta_{p, \ell, j}}{\ell^j}$ is an integer. Since $0\leq \beta_{p, \ell, j} \leq \ell^j - 1$, we observe that 
\begin{equation} \label{beta_delta}
\frac{\beta_{p,\ell,j} + \delta_{p,\ell,j}}{\ell^j} = \left\lceil\delta_{p,\ell,j}/\ell^j \right\rceil.
\end{equation}
Next, with $\Delta = \gcd(24, p+1)$, we set $D = \frac{24}{\Delta}$, and we define $t_{p,\ell,j}$ by
\begin{equation}\label{eqn_beta}
    \frac{\beta_{p,\ell,j} + \delta_{p,\ell,j}}{\ell^{j}} = \left( \frac{p + 1}{\Delta} \right) t_{p.\ell,j} + r, \ \ \ 0 \leq r < \frac{p + 1}{\Delta}.
\end{equation}
From \eqref{beta_delta} and \eqref{eqn_beta}, it follows that 
\begin{equation}\label{defn_t}
t_{p,\ell,j} = \left\lfloor \frac{\left\lceil \delta_{p, \ell, j}/{\ell^j}  \right\rceil}{(p + 1) / \Delta}  \right\rfloor.
\end{equation}
We also define
\begin{equation} \label{eq:lambda_def}
\lambda_{\ell,j} = \ell^{j} - 1 + \ell^{j - 1}(\ell - 1).
\end{equation}
For simplicity, we omit the subscripts on the quantities above in everything that follows when the context is clear. Our first theorem is the analogue of Theorem \ref{AB} for the function~$p_{[1, p]}(n)$.

	\begin{theorem} \label{MT1}
		Assume the notation above. Let $\ell$ and $p$ be distinct primes with $\ell\geq 5$, and let $j$ be a positive integer. Then there exists a modular form $H(z) = H_{p, \ell, j}(z) \in M_{\lambda - Dt}\left(\Gamma_{0}(p), \left(\frac{-p}{\bigcdot}\right)^{Dt} \right)$ with integer coefficients such that
		\begin{equation}\label{defn_H}
			\sum_{m = 0}^{\infty} p_{\left[1, p \right]} (\ell^{j}m - \delta)  \: q^{m - \frac{\ell^j \left(\frac{p + 1}{\Delta} \right)}{D}} \equiv (\eta(z)\eta(pz))^{Dt - \ell^j}\,H(z) \Mod{\ell^j}.
		\end{equation}
	\end{theorem}
    
 \noindent
 {\bf Remarks.}
 \begin{enumerate}
     \item Applying the operator $V_{D} : z \mapsto Dz$ to both sides of \eqref{defn_H}, we deduce that
     \begin{equation*}
         \sum_{n \equiv - \ell^{j} \left( \frac{p + 1}{\Delta} \right) \Mod{D}} p_{[1,p]} \left( \frac{\ell^j n + \frac{p + 1}{\Delta}}{D}   \right) q^n \equiv (\eta(Dz) \eta(Dpz))^{Dt - \ell^j} H(Dz) \Mod{\ell^j},
     \end{equation*}
     where the modular form on the right side lies in $M_{\ell^{j - 1}(\ell - 1) - 1}\left(\Gamma_0(pD^2),\left(\frac{-p}{\bigcdot}\right)\right)$.
     \item Elementary arguments using properties of the floor and ceiling functions show that the quantity $t$ in \eqref{defn_t} simplifies according to the size of $\frac{p + 1}{24}$ relative to $\ell^j$ as follows:
\renewcommand{\arraystretch}{1.5} 
\setlength{\tabcolsep}{12pt} 
\begin{table}[ht]
  \centering
\begin{tabular}{|c|c|}
\hline
Conditions & $t$ \\ 
\hline
$\frac{p + 1}{24}\geq \ell^j$
& \makecell{$ \big\lfloor \frac{\ell^j}{D} \big\rfloor$ \ \text{if} \ $D > 1$ \\[5 pt]
 $\ell^j - 1 \ \text{if} \ D = 1$}  \\ 
\hline
$1\leq \frac{p + 1}{24} < \ell^j$ & $\big\lfloor \frac{\ell^j}{D} \big\rfloor$ \\
\hline
\makecell{$\frac{p + 1}{24} < 1, \: p + 1 \mid 24$ \\ [5 pt] ($p\leq 11$)}
& $\big\lfloor \frac{\ell^j}{D} \big\rfloor + 1$ \\
\hline
\makecell{$\frac{p + 1}{24} < 1, \: p + 1\nmid 24$ \\ [5 pt] ($p\in \{13, 17, 19\}$)} & \makecell{$\big\lfloor \frac{\ell^j}{D} \big\rfloor + 1 \ \text{if} \ \ell \equiv -1\Mod{D} \ \text{and} \ (p, \ell) \neq (19, 5)$ \\ [5 pt]
$\big\lfloor \frac{\ell^j}{D} \big\rfloor$ \ \text{otherwise}} \\
\hline
\end{tabular}
 \vskip.15in
  \caption{Values of $t$}
  \label{tab:j_even}
\end{table}
\vspace{-.15in}
     \item For all $\alpha\in \Z$, we let $\overline{\alpha} \equiv \alpha\Mod{D}$ with $0\leq \overline{\alpha} \leq D - 1$.  In view of the previous remark, the exponent $Dt - \ell^j$ on $\eta(z)\eta(pz)$ in \eqref{defn_H} simplifies as 
     \begin{equation*}
         Dt - \ell^j = \begin{cases}
         -\overline{\ell^j}, & t = \big\lfloor \frac{\ell^j}{D} \big\rfloor, \\
         \overline{D - \ell^j}, & t = \big\lfloor \frac{\ell^j}{D} \big\rfloor + 1;
         \end{cases}
     \end{equation*}
     hence, it has absolute value bounded by $D - 1$.
     Furthermore, it follows that the exponent is non-negative if and only if any the following holds:
     \begin{enumerate}
         \item $1\leq \frac{p + 1}{24} < \ell^j$ and $D = 1$ (i.e., $p\equiv -1\Mod{24}$), 
         \item $p + 1\mid 24$,
         \item $p\in \{13, 17, 19\}$, $\ell \equiv -1\Mod{D}$ and $(p, \ell) \neq (19, 5)$.
     \end{enumerate}
     \item With $\lambda$ as in \eqref{eq:lambda_def}, we observe that the generating function for $p_{[1, p]}(\ell^jm - \delta)$ in \eqref{defn_H} has weight $\lambda - \ell^j = \ell^{j - 1}(\ell - 1) - 1 < \ell^{j - 1}(\ell - 1)$.  Since modular forms modulo $\ell^j$ have weight grading with values in $\Z/(\ell^{j - 1}(\ell - 1))\Z$, if the generating function is nonzero modulo $\ell^j$, then its filtration (minimal weight) modulo $\ell^j$ is $\ell^{j - 1}(\ell - 1) - 1$.  When $j = 1$, Theorem~1.1 of \cite{sinickcongruences} implies that the generating function is nonzero modulo $\ell$ for $\ell\geq 5$ and therefore, has filtration $\ell - 2$.  
 \end{enumerate}

\medskip

\noindent
We supply a simple example to illustrate the theorem. 

\medskip

\noindent
{\bf Example.}
Let $p = 5$, $j = 1$, and $\ell\geq 7$.  Then we have $\Delta = \gcd(6, 24) = 6$, $D = \frac{24}{\Delta} = 4$, and $\lambda = 2\ell - 2$.  The second remark following Theorem \ref{MT1} gives $t = \lfloor\ell/4\rfloor + 1$, while the second and third give 
\begin{equation*}
  Dt - \ell = \begin{cases}
    3, & \ell \equiv 1\Mod{4}, \\
    1, & \ell \equiv 3\Mod{4}; 
    \end{cases}
\ \ \text{and} \ \ 
  \lambda - Dt = \begin{cases}
      \ell + 5, & \ell \equiv 1\Mod{4}, \\
      \ell + 3, & \ell \equiv 3\Mod{4}.
  \end{cases} 
\end{equation*}
Theorem \ref{MT1} implies that
\[
\sum_{n = 0}^{\infty}p_{[1,5]}\left(\frac{\ell n +1}{4}\right)q^n \equiv (\eta(4z)\eta(20z))^{Dt - \ell}H(4z)\Mod{\ell}
\]
with $H(z)\in M_{\lambda - Dt}(\Gamma_0(5))$.  For primes $\ell \in \{7, 11, 13, 17\}$, we use the basis 
\begin{equation}\label{basis}
\left\{f_{0}(z) = \frac{\eta(z)^{10}}{\eta(5z)^{2}} = 1 + \cdots, ~ f_{1}(z) = \eta(z)^{4}~\eta(5z)^{4} = q + \cdots, ~ f_{2}(z) = \frac{\eta(5 z)^{10}}{\eta(z)^{2}} = q^2 + \cdots\right\}
\end{equation}
for $M_{4}(\Gamma_{0}(5))$ to compute $H(z)$ explicitly as in the following table.  
\renewcommand{\arraystretch}{1.5} 
\setlength{\tabcolsep}{12pt} 
\begin{table}[ht]
  \centering
\begin{tabular}{|c|c|}
\hline
$\ell$ & $H(z)$ \\ 
\hline
$7$ & $2 f_{0} + f_{1} + 5 f_{2}$ \\
 \hline
$11$ & $3 f_{0}^{2} + 7 f_{0}f_{1} + 3 f_{1}^{2} + 6 f_{1} f_{2} + 4 f_{2}^{2}$ \\
 \hline
$13$ & $12 f_{0}^{2} + 2 f_{0}f_{1} + 6 f_{1}^{2} + 3 f_{1} f_{2} + f_{2}^{2}$\\
 \hline
$17$ & $10 f_{0}^{3} + 16 f_{0}^{2} f_{1} + 7 f_{0}^{2} f_{2}  + 13 f_{0}f_{1}f_{2} + 8 f_{0} f_{2}^{2} + 15 f_{1} f_{2}^{2} + f_{2}^{3}$\\
 \hline
\end{tabular}
 \vskip.15in
  \caption{$H(z)$ in Theorem \ref{B1} for $p = 5, \: j = 1, \: \ell \in \{7, 11, 13, 17\}$}
\end{table}

\vspace{-.15in}
Our second theorem is the analogue of Theorem \ref{Yang} for the function $p_{[1,p]}(n)$.  
	\begin{theorem} \label{E5}
  Let $\ell$, $m$, and $p$ be distinct primes with $\ell$ and $m\geq 5$ and $p\in \{2, 3, 5\}$, and let $j\geq 1$ be an integer.  Then there exist explicitly computable integers $J$ and $N\geq 1$
  such that the following congruences hold.  
  \begin{enumerate} 
  \item For all positive integers $v$ and $n$ with $m\nmid n$, we have
		\begin{gather} \label{E13}
			p_{\left[1, p \right]} \left( \frac{\ell^{j} \: m^{2vJ - 1} \: n + 1}{D}  \right) \equiv 0 \Mod{\ell^{j}}.
		\end{gather}
  \item For all nonnegative integers $w$ and $n$, we have 
		\begin{gather} \label{E14}
			p_{\left[ 1, p  \right]} \left( \frac{\ell^{j} \: m^{w}\: n + 1}{D}  \right) \equiv  p_{\left[ 1, p  \right]} \left( \frac{\ell^{j} \: m^{2N + w} \: n + 1}{D} \right)  \Mod{\ell^{j}}.  
		\end{gather} 
  \end{enumerate}
	\end{theorem}

\noindent
We furnish some examples of the first part of the theorem.

\medskip

\noindent
{\bf Example 1.}  Let $\ell = 13$, $m = 7$, $p = 5$, and $j = 1$.  When $p = 5$, we have $D = 4$.  We find that the first part of the theorem holds with $J = 1190$: for all $v\geq 1$ and all $n\geq 1$ with $7\nmid n$, we have
\begin{equation}
p_{[1,5]}\left(\frac{13\bigcdot 7^{2380v - 1}n + 1}{4}\right) \equiv 0 \Mod{13}. \label{ex1}
\end{equation}

\medskip

\noindent
{\bf Example 2.}  Let $\ell = 7$, $m = 23$, $p = 3$, and $j = 2$.  When $p = 3$, we have $D = 6$.  We find that the first part of the theorem holds with $J = 1176$: for all $v\geq 1$ and all $n\geq 1$ with $23\nmid n$, we have 
\begin{equation}
p_{[1,3]}\left(\frac{7^2\bigcdot 23^{2352v - 1}n + 1}{6}\right)\equiv 0\Mod{7^2}.
\label{ex2}
\end{equation}

\noindent
{\bf Remark.} We provide an explicit description for $J$ and $N$ below and in the proof of Theorem~\ref{E5} in Section~3.  At the end of that section, we give details on the computations for our two examples.

\medskip

Before we proceed with the proofs of Theorems \ref{MT1} and \ref{E5}, we briefly discuss the interplay between them. It turns out that Theorems \ref{Yang} and \ref{E5} ultimately follow from the fact that the relevant generating functions modulo $\ell^j$ in Theorems \ref{ahlgren_boylan} and \ref{MT1} belong to subspaces invariant under Hecke operators with square index.  For details on Hecke operators in integer weight, see Section 2. In particular, the first remark following Theorem \ref{MT1} implies that, modulo~$\ell^j$, we~have 
\begin{equation*}
\sum_{n \equiv - \ell^{j} \left( \frac{p + 1}{\Delta} \right) \Mod{D}} p_{[1,p]} \left( \frac{\ell^j n + \frac{p + 1}{\Delta}}{D}   \right) q^n
\in \mathcal{A}_{p,y,k,\chi},
\end{equation*}
where 
\begin{equation}\label{Aspace}
\mathcal{A}_{p,y,k,\chi} = \{(\eta(Dz)\eta(Dpz))^{y}\,H(Dz) : H(z)\in M_k(\Gamma_0(p),\chi)\},
\end{equation}
$y = Dt - \ell^j$, $k = \lambda - Dt$ and $\chi = \left( \frac{-p}{\bigcdot} \right)^{Dt}$. 
We observe that $\mathcal{A}_{p, y, k, \chi} \cong M_{k}(\Gamma_0(p),\chi)$ as $\mathbb{C}$-vector spaces and that when $y\geq 0$ as in the third remark following Theorem \ref{MT1}, we have
\[
\mathcal{A}_{p, y, k, \chi}\subseteq M_{k + y}\left(\Gamma_0(pD^2),\left(\frac{-p}{\bigcdot}\right)\right).
\]
The following theorem is the special case we require of a much more general theorem proved in~\cite{warnockthesis}.      
\begin{theorem}\label{warnock_boylan}
Assume the notation above with $y = Dt - \ell^j\geq 0$ and $k = \lambda - Dt$, and let $p$ and $m$ be distinct primes with $p\in \{2, 3, 5\}$ and $m\geq 5$.  Then the integer weight Hecke operator $T_{m^2}$ maps the subspace $\mathcal{A}_{p, y, k, \chi}\subseteq M_{k + y}\left(\Gamma_0(pD^2),\left(\frac{-p}{\bigcdot}\right)\right)$ to itself.       
\end{theorem}
The work \cite{warnockthesis} precisely describes the subspace of $M_{k + y}\left(\Gamma_0(pD^2), \left(\frac{-p}{\bigcdot}\right)\right)$ to which an integer weight Hecke operator $T_n$ maps $\mathcal{A}_{p, y, k, \chi}$; it is generally {\it not} to $\mathcal{A}_{p, y, k, \chi}$ itself.  
  The subspace $\mathcal{A}_{p, y, k, \chi}$ has small dimension relative to the ambient space $M_{k + y}\left(\Gamma_0(pD^2), \left(\frac{-p}{\bigcdot}\right)\right)$.    Standard formulas (see Section 7.4 of \cite{Cohen2017ModularFA} for example) give dimensions of the relevant spaces in Table III~below.

\renewcommand{\arraystretch}{1.5} 
\setlength{\tabcolsep}{12pt} 
\begin{table}[ht]
  \centering
\begin{tabular}{|c|c|c|}
\hline
$p$ & $\text{dim}(\mathcal{A}_{p, y, k, \chi})$ & $\text{dim}(M_{s}\left(\Gamma_0(pD^2),\left(\frac{-p}{\bigcdot}\right)\right)$ \\
\hline
$2$ & $\left\lfloor\frac{k}{4}\right\rfloor + 1$ & $16 (k + y) - 8$ \\
\hline
$3$ & $\left\lfloor\frac{k}{3}\right\rfloor + 1$ & $18 (k + y) - 9$  \\
\hline
$5$ & $2\left\lfloor\frac{k}{4}\right\rfloor + 1$ & $12 (k + y) - 6$. \\
\hline
\end{tabular}
 \vskip.15in
  \caption{Dimension Comparison}
\vspace{-.35in}
\end{table}
Theorem \ref{warnock_boylan} allows us to explicitly describe the integers $J$ and $N$ in Theorem \ref{E5}.  To see this, we let $d =~d_{p,\ell, j}$ be the dimension of $\mathcal{A}_{p, y, k, \chi}$ as in Table III, and we let $\{f_1,\dots, f_d\}$ be a $\Z$-basis for the $\Z$-module $\mathcal{A}_{p, y, k, \chi}\cap~\Z[[q]]$.  With $\vec{f}=\langle f_1,\dots, f_d\rangle^{\text{t}}$, the $T_{m^2}$-invariance of $\mathcal{A}_{p, y, k, \chi}$ implies that
there exists $M\in \text{Mat}_{d\times d}(\Z)$ such that $\vec{f}\mid T_{m^2} = M\vec{f}$.  We let ${I}_{d}$ and ${0}_{d}$ denote the identity and zero matrices in $\text{Mat}_{d\times d}(\Z)$, and we let 
\begin{equation}
A = \begin{pmatrix} M - \left(\frac{-p}{m}\right) m^{k + y - 1} I_{d} & -m^{2(k + y) - 2} {I}_{d} \\ I_{d}  & {0}_{d} \end{pmatrix}\in \text{GL}_{\,2d}(\Z). \label{A_matrix}
\end{equation}
In this notation, the integer $J$ in Theorem \ref{E5} arises as the order of $A$ in $\text{PGL}_{2d}(\Z/\ell^j\Z)$, and the integer $N$ is the order of $A$ in $\text{GL}_{2d}(\Z/\ell^j\Z)$.  

We also note that, modulo $\ell^j$, the generating functions in Theorem \ref{ahlgren_boylan} lie in subspaces of half-integral weight holomorphic modular forms of type $S_{r, w} = \{\eta(24z)^rF(24z) : F(z) \in M_w(\Gamma_0(1))\}$.  The subspaces $S_{r,w}$ are invariant under the half-integral weight Hecke operators.  The integers $K$ and $M$ in Theorem \ref{Yang} arise exactly as $J$ and $N$ do in Theorem \ref{E5}: as orders of matrices in $\text{PGL}_{2t}(\Z/\ell^j\Z)$ and $\text{GL}_{2t}(\Z/\ell^j\Z)$ which encode the action of Hecke operators with index $m^2$ on a $\Z$-basis for the $\Z$-module $S_{r, w}\cap\Z[[q]]$.  Since the Hecke operators $T_{m^2}$ are different in integral and half-integral weights, the matrices needed to define $K$ and $M$ in Theorem \ref{Yang} and $J$ and $N$ in Theorem \ref{E5} have different general forms.  

The outline for the rest of the paper is as follows.  In Section 2, we provide necessary background on modular forms.  In Section 3, we prove our main theorems, Theorem \ref{MT1} and \ref{E5}.  We also provide explicit numerical examples of our theorems. 

\vskip.15in

\noindent
{\bf Acknowledgment.}  The authors thank the referee for their helpful comments and suggestions and especially for identifying parts of the paper, notably the proof of Theorem \ref{ord_vanish_G}, which needed  more clarification than the authors provided initially.
 
 
\section{Background on Modular Forms}

For background on modular forms, one may consult \cite{Cohen2017ModularFA}, for example.  

\subsection{Operators on Spaces of Modular Forms}
As operators on spaces of modular forms play a central role in our work, we recall definitions of the operators that we use, starting with the slash operator.  Throughout this section, we let $N$ and $k$ be integers with $N\geq 1$, and we let $\psi$ be a Dirichlet character modulo $N$.  For functions $f$ on the complex upper half-plane and for all $\gamma = \begin{pmatrix} a & b \\ c & d\end{pmatrix} \in \text{GL}_2^{+}(\mathbb{R})$, we define the slash operator in weight $k$ by 
\begin{equation} \label{SLO}
\left(f\mid _k\gamma\right)(z) = \left(\text{det}\gamma\right)^{\frac{k}{2}}(cz+d)^{-k}f(\gamma z).
\end{equation}
With $q = e^{2\pi iz}$, we let $f = \sum a(n)q^n \in M_k(\Gamma_0(N),\psi)$.
For all integers $d\geq 1$, we define the $U$- and $V$-operators by 
\begin{align}
& f\mid V_d = d^{-\frac{k}{2}}f\mid_k\begin{pmatrix} d & 0 \\ 0 & 1 \end{pmatrix}, \label{VOP} \\
& f\mid U_d = d^{\frac{k}{2} - 1}\sum_{j = 0}^{d - 1}f\mid_k \begin{pmatrix} 1 & j \\ 0 & d\end{pmatrix}. \label{UOP}
\end{align}
On $q$-expansions, we have 
\begin{equation} \label{UVq}
f\mid V_d = \sum a(n)q^{dn}, \ \ \ 
f\mid U_d = \sum a(dn)q^n.  
\end{equation}
These operators map spaces of modular forms as follows: 
\begin{align}
& V_d : M_k(\Gamma_0(N),\psi) \longrightarrow M_k(\Gamma_0(dN),\psi), \label{vmap} \\
& U_d : M_k(\Gamma_0(N),\psi)\longrightarrow \begin{cases} M_k(\Gamma_0(dN), \psi) & d\nmid N, \\ M_k(\Gamma_0(N), \psi) & d\mid N.\label{umap}\end{cases}
\end{align}

\noindent
We now turn to Hecke operators.  For all integers $n\geq 1$ we define the Hecke operator $T_n = T_{n, k,\psi}$ on $M_k(\Gamma_0(N),\psi)$~by
\begin{equation}\label{E4}
T_{n} = \sum_{d \mid n} \psi(d) d^{k - 1} V_{d} \circ U_{n/d}.
\end{equation}
When $\ell$ is prime, $f\mid T_{\ell}$ has $q$-expansion
\[
f \mid T_{\ell} = \sum\left(a(\ell n) + \psi(\ell){\ell}^{k - 1}a\left(\frac{n}{\ell}\right)\right)q^n,
\]
where $a\left(\frac{n}{\ell}\right) = 0$ when $\ell\nmid n$.  The Hecke operators preserve 
$M_k(\Gamma_0(N),\psi)$ and its subspace of cusp forms.  

We require further operators defined using \eqref{SLO}.  With 
\begin{align}
H_N = \begin{pmatrix} 0 & -1 \\ N & 0\end{pmatrix}, \label{Fricke}
\end{align}
we define the Fricke involution $f\mapsto f\mid_k H_N$, which maps $M_k(\Gamma_0(N),\psi)
\rightarrow M_k(\Gamma_0(N),\overline{\psi})$.  When $\gcd(\ell, N/\ell) = 1$, we define 
\begin{equation} \label{ATI}
W_{\ell}^{N} = \begin{pmatrix} \ell & a \\ N & b\ell\end{pmatrix},
\end{equation}
where $a$, $b\in \Z$ and $\text{det}(W_{\ell}^N) = \ell$.  If $\psi$ is defined modulo $N/\ell$, then Lemma 2 of \cite{Li1975NewformsAF} implies that $\mid_k W_{\ell}^{N}$ maps $M_k(\Gamma_0(N),\psi)$ to itself and does not depend on the choice of $a$ and $b$.  
Keeping our hypotheses on $\ell$ and $\psi$, we define the trace operator on $M_k(\Gamma_0(N),\psi)$ by
\begin{equation}\label{trace}
\text{Tr}_{N/\ell}^{N}(f) = f + \psi(\ell){\ell}^{1 - \frac{k}{2}}f\mid_k W_{\ell}^{N}\mid U_{\ell}.
\end{equation}
\noindent
We also recall the notion of twisting by a Dirichlet character $\chi$ modulo $M\geq 1$:  
\begin{equation} \label{twist}
f \otimes \chi = \sum \chi(n)a(n)q^n \in M_k(\Gamma_0(NM^2), \psi\chi^2).  
\end{equation}

We record some further properties of the operators we defined in this section.  
We will use the following elementary property of the $U$-operator on $q$-series.  
\begin{proposition}\label{U_factor}
Let $d\geq 1$ in $\mathbb{Z}$, and let $f(q)$, $g(q)\in \mathbb{C}[[q]]$.  Then we have 
\[
(f(q)g(q^d))\mid U_{d} = f(q)\mid U_{d} \bigcdot g(q).
\]  
\end{proposition}
\noindent The trace and $U$-operators play a role in lowering levels.  
\begin{proposition}[Lemmas 1 and 3 of \cite{Li1975NewformsAF}]\label{level_lower}
Let $N\geq 1$ be an integer, let $\ell$ be prime with $\ell\mid N$, and suppose that $\psi$ is a character modulo $N/\ell$.
\begin{enumerate}
\item 
If $\ell^2\mid N$, then we have
$U_{\ell} : M_k(\Gamma_0(N),\psi)\rightarrow M_k(\Gamma_0(N/\ell), \psi)$.
\item 
If $\ell \mid N$ and $\gcd(\ell, N/\ell) = 1$, then we have $\text{Tr}_{N/\ell}^{N} : M_k(\Gamma_0(N), \psi)\rightarrow M_k(\Gamma_0(N/\ell), \psi)$.
\end{enumerate}
\end{proposition}
\noindent
We need a result on the commutation of $U$ as in \eqref{UOP} and $W$ as in \eqref{ATI}.
\begin{proposition}[Proposition 1.3 of \cite{atkin-li}] \label{atkin-li}
Let $f\in M_k(\Gamma_0(N),\psi)$ and let $\ell\neq p$ be primes such that $\ell, p\mid N$ and 
$\gcd(\ell, N/\ell) = 1$. Then we have
\begin{equation*}
 f \mid U_p\mid W_{\ell}^{N} = \psi_{\ell}(p) \: f \mid W_{\ell}^{N} \mid U_{p},   
\end{equation*}
where $\psi_{\ell}$ and $\psi_{N/\ell}$ are characters with moduli $\ell$ and $N/\ell$ and satisfy $\psi = \psi_{\ell}\psi_{N/\ell}$.
\end{proposition}
\noindent 
We also need a result on the commutation of $V$ as in \eqref{VOP} and $W$ as in \eqref{ATI}.
\begin{proposition} \label{Commutativity}
Let $\ell$ be prime, let $k$ and $t$ be integers with $t\geq 1$ and $\ell\nmid t$, let $\psi$ be a Dirichlet character modulo $\ell$, and let $f \in M_k(\Gamma_0(\ell),\psi)$.  Then we have 
\begin{enumerate}
\item 
$f\mid V_t\mid_k W_{\ell}^{t\ell} = f\mid_k H_{\ell}\mid V_t$ and 
\item 
$f\mid_k W_{\ell}^{t\ell} = \psi(t) f\mid_k H_{\ell}$.  
\end{enumerate}
\end{proposition} 
\begin{proof}
Using \eqref{VOP} and \eqref{ATI}, we compute 
\begin{align*}
f\mid V_{t}\mid_k W_{\ell}^{t\ell}& =t^{-k/2}f\Big|_k
\left[\begin{pmatrix} t & 0 \\ 0 & 1 \end{pmatrix}\begin{pmatrix}   
\ell & a \\ t\ell & b \ell \end{pmatrix}\right] 
= t^{-k/2}f \Big|_k \begin{pmatrix} t \ell & at \\ t\ell & b \ell \end{pmatrix} \\
& = t^{-k/2}f \Big|_k \left[\begin{pmatrix} -at & 1 \\ -b \ell & 1 \end{pmatrix} \begin{pmatrix} 0 & -1 \\ \ell & 0 \end{pmatrix} 
\begin{pmatrix} t & 0 \\ 0 & 1 \end{pmatrix} \right] 
= t^{-k/2}f \Big|_k 
\left[\begin{pmatrix} 0 & -1 \\ \ell & 0 \end{pmatrix} 
\begin{pmatrix} t & 0 \\ 0 & 1 \end{pmatrix} \right] \\ 
&=  f \mid_k H_{\ell} \mid V_{t}.
\end{align*}
Similarly, we use \eqref{Fricke} and \eqref{ATI} to compute 
\begin{align*}
f\mid_k W_{\ell}^{t\ell} = f\mid_k\begin{pmatrix}\ell & a \\ t\ell & b\ell\end{pmatrix} = f\mid_k\begin{pmatrix} -a & 1 \\-b\ell & t\end{pmatrix} \begin{pmatrix} 0 & -1 \\ \ell & 0 \end{pmatrix} = \psi(t) f\mid_k H_{\ell}.
\end{align*}
\end{proof}

\subsection{Some Properties of Eta-Quotients}
As the Dedekind eta-function \eqref{eta1} plays a central role in our work, we review some of its properties here. It is a modular form of weight 1/2 on $\text{SL}_{2}(\mathbb{Z})$ with respect to a multiplier $\epsilon_{a,b,c,d}$ with values in the group of $24$th roots of unity: for all $z\in \mathbb{H}$ and for all $\gamma = \begin{pmatrix} a & b \\ c & d \end{pmatrix} \in \text{SL}_2(\Z)$, we have 

\begin{align} \label{ETA}
\eta(\gamma z) = \epsilon_{a,b,c,d}  (cz + d)^{1/2}  \eta(z). 
\end{align}
For the precise definition of $\epsilon_{a, b, c, d}$, see Theorem 5.8.1 of \cite{Cohen2017ModularFA}.  For example, when $\gamma = \begin{pmatrix} 0 & -1 \\ 1 & 0\end{pmatrix}$, we have 
\begin{align}
\eta\left(-\frac{1}{z}\right) = \sqrt{\frac{z}{i}}\,\eta(z). \label{eta_S}
\end{align}

The eta-function is a building block for modular forms.  Let $N\geq 1$.  An eta-quotient of level $N$ is a function of the form
\begin{equation}
\label{etaquot}
f(z) = \prod_{\delta \mid N} \eta(\delta z)^{r_{\delta}},
\end{equation}
where $\delta$ and $r_{\delta}$ are integers with $\delta \geq 1$. 
The following proposition (see for example, Proposition 5.9.2 of~\cite{Cohen2017ModularFA}) gives criteria for an eta-quotient to be a weakly holomorphic modular form. 

\begin{proposition}\label{GHN}
Let $f(z)$ be an eta-quotient of level $N$ as in \eqref{etaquot} with $k = \displaystyle{\frac{1}{2} \sum_{\delta\mid N}r_{\delta}\in \Z}$.  Suppose that $f(z)$ satisfies 
\[
\sum_{\delta \mid N}\delta r_{\delta} \equiv 0 \Mod{24} \ \ \text{and} \ \ 
N\sum_{\delta\mid N} \frac{r_{\delta}}{\delta} \equiv 0\Mod{24}.  
\]
Then for all $\gamma = \begin{pmatrix} a & b \\ c & d\end{pmatrix} \in \Gamma_0(N)$, we have 
$f(\gamma z) = \chi(d)(cz + d)^k f(z)$, where $\chi$ is defined by $\chi(d) = \left(\frac{(-1)^ks}{d}\right)$ with $s = \displaystyle{\prod_{\delta\mid N}\delta^{r_{\delta}}}$.
\end{proposition}
\noindent
We also require the formula (see for example, Proposition 5.9.3 of \cite{Cohen2017ModularFA}) for the order of vanishing of an eta-quotient at cusps.  We let $\frac{c}{d}\in \mathbb{P}^{1}(\mathbb{Q})$ where $d \mid N$.  We denote its equivalence class with respect to $\Gamma_0(N)$ by $\left[\frac{c}{d}\right]_{N}$ and its width with respect to $\Gamma_0(N)$ by $h_{d, N} = \frac{N}{\gcd(d^2, N)}$.   

\begin{proposition} \label{CE}
Let $N$, $c$, and $d$ be positive integers with $\gcd(c, d) = 1$,
and let $f(z)$ be an eta-quotient of level $N$ as in \eqref{etaquot}.  Then the order of vanishing of $f(z)$ at the cusp $\left[\frac{c}{d}\right]_{N}$ with respect to the local variable $q_{h_{d, N}} = q^{\frac{1}{h_{d,N}}}$ is given by
\[
\operatorname{ord}_{\left[\frac{c}{d}\right]_{N}}(f) = \frac{h_{d, N}}{24 } \sum_{\delta \mid N} \frac{(\textup{gcd}(d, \delta))^{2} \: r_{\delta}}{\delta}.
\]
\end{proposition}


\section{Proofs of Theorems \ref{MT1} and \ref{E5}}
\subsection{Setup for the proof of Theorem \ref{MT1}} \label{reduction}
We let $\ell$ and $p$ be distinct primes with $\ell\geq 5$, and we let $j\geq 1$ be an integer.  We define
\begin{equation} \label{FVAL}
        f_{\ell, j}(z) = \frac{\eta(\ell^{j}z)^{\ell^{j}}}{\eta(z)} \in M_{\frac{\ell^{j} - 1}{2}} \left(\Gamma_{0} (\ell^{j}) , \left(\frac{\bigcdot}{\ell} \right)^{j} \right).
\end{equation}
Recalling the generating function for $p_{[1, p]}(n)$ in \eqref{genfun}, the definition of $\delta$ in \eqref{B2}, and the definitions $\Delta = \gcd(24, p + 1)$ and $D = \frac{24}{\Delta}$, we first observe that 
\begin{align}
(f_{\ell, j}(z)f_{\ell, j}(pz))\mid U_{\ell^j} 
& = \left(q^{\delta}\prod_{n = 1}^{\infty}\frac{(1 - q^{\ell^jn})^{\ell^{j}}(1 - q^{\ell^jpn})^{\ell^{j}}}{(1 - q^n)(1 - q^{pn})}\right)\mid U_{\ell^j} \notag \\
& = \left(\prod_{n = 1}^{\infty}(1 - q^{\ell^{j} n})^{\ell^j}(1 - q^{\ell^{j} p n})^{\ell^j}\right)\mid U_{\ell^j}\bigcdot
\left(\sum_{s = 0}^{\infty}p_{[1, p]}(s)q^{s +  \delta}\right)\mid U_{\ell^j} \notag\\
& = \eta(z)^{\ell^j}\eta(pz)^{\ell^j}\sum_{r = 0}^{\infty}p_{[1,p]}(\ell^jr - \delta)q^
{r - \frac{\ell^j\left(\frac{p + 1}{\Delta}\right)}{D}},
\label{U_on_f}
\end{align}
where we used Proposition \ref{U_factor} for the second equality and \eqref{vmap} for the third. 

\subsection{Key inputs to the proof of Theorem \ref{MT1}.}

We next prove Theorems \ref{T1} and \ref{ord_vanish_G}, whose truth implies Theorem \ref{MT1}.  Theorem \ref{T1} 
explicitly constructs a form on $\Gamma_0(p)$ congruent modulo~$\ell^j$ to the form $(f_{\ell, j}(z)f_{\ell, j}(pz))\mid U_{\ell^j}$ in \eqref{U_on_f}, and Theorem \ref{ord_vanish_G} asserts that this form has orders at cusps greater than or equal to $\lceil \delta/\ell^j\rceil$. 

\begin{theorem} \label{T1}
Let $\ell$ and $p$ be distinct primes with $\ell\geq 5$, let $j\geq 1$ be an integer, and let 
\[
\lambda_{\ell, j} = \ell^{j} - 1 + \ell^{j - 1}(\ell - 1)
\]
be as in \eqref{eq:lambda_def}.
Then there exists a modular form $G_{p,\ell,j}(z) \in M_{\lambda_{\ell, j}}(\Gamma_{0}(p)) \cap \mathbb{Z}[[q]]$ such that
\begin{gather*} \label{existence of H}
(f_{\ell, j}(z) f_{\ell, j}(pz)) \mid U_{\ell^{j}} \equiv G_{p,\ell,j}(z) \Mod{\ell^{j}}. 
\end{gather*}
\end{theorem}

\medskip

The proof of Theorem \ref{T1} requires
\begin{align}
    h_{\ell,j}(z) = \left( \frac{\eta(z)^{\ell}} {\eta(\ell z)} \right)^{\ell^{j - 1}} = h_{\ell, 1}(z)^{\ell^{j - 1}}\in M_{\frac{\ell^{j - 1}(\ell - 1)} {2}} \left(\Gamma_{0}(\ell),   \left( \frac{\bigcdot}{\ell} \right) \right). \label{hdef}
\end{align}
We note that $h_{\ell, j}(z)\equiv 1\Mod{\ell^{j}}$, which follows by induction on $j$ using the binomial theorem.  With $f_{\ell, j}$ as in \eqref{FVAL}, we use \eqref{Fricke} and \eqref{eta_S} to compute
\begin{gather}
h_{\ell,j}(z) \mid_{\frac{\ell^{j - 1}(\ell - 1)} {2}} H_{\ell} = \ell^{\frac{\ell^{j} + \ell^{j - 1}}{4}} (-i)^{\frac{\ell^{j} - \ell^{j - 1}}{2}} f_{\ell, 1}(z)^{\ell^{j - 1}},
\label{hH}
\end{gather}
and we define 
\begin{gather} \label{LP}
    g_{p,\ell,j}(z) = (f_{\ell, j}(z)f_{\ell, j}(pz))\mid U_{\ell^{j - 1}}\bigcdot h_{\ell, j}(z)h_{\ell, j}(pz). 
\end{gather}
We observe that $g_{p, \ell, j}(z)$ has integer coefficients since $f_{\ell, j}(z)$ and $h_{\ell, j}(z)$ do.  Therefore, since $h_{\ell,j}(z) \equiv 1\Mod{\ell^j}$, it follows from \eqref{LP} that 
\begin{align}
g_{p, \ell, j}(z)\mid U_{\ell} \equiv (f_{\ell,j}(z)f_{\ell,j}(pz))\mid U_{\ell^j} \Mod{\ell^j}. \label{g_cong}
\end{align}
Using \eqref{vmap}, the first part of Proposition \ref{level_lower}, \eqref{FVAL}, and \eqref{hdef}, we find that $g_{p,\ell,j}(z)\in M_{\lambda}(\Gamma_{0}(\ell p))$, where $\lambda = \lambda_{\ell, j}$ as in Theorem \ref{T1}. We apply the trace operator \eqref{trace} to $g_{p,\ell,j}(z) \mid_{\lambda} W_{\ell}^{p \ell}$, we observe that $W_{\ell}^{p\ell}\circ W_{\ell}^{p\ell} = 1$, and we use the second part of Proposition \ref{level_lower} to obtain
\begin{gather*} \label{TR}
    \text{Tr}_{p}^{ \ell p} \left(g_{p,\ell,j}(z) \mid_{\lambda} W_{\ell}^{p \ell} \right) =  g_{p,\ell,j}(z) \mid_{\lambda} W_{\ell}^{p \ell} + \ell^{1 - \frac{\lambda}{2}} g_{p,\ell,j}(z)  \mid U_{\ell} \in M_{\lambda}(\Gamma_{0}(p)).
\end{gather*}
Since the map $F(z)\mapsto F(z)\mid_k W_{\ell}^{p\ell}$ preserves the field of rationality of $F(z)\in M_k(\Gamma_0(p\ell))$, we find that $g_{p, \ell, j}(z)\mid_{\lambda} W_{\ell}^{p\ell}$ has rational coefficients; hence, $\text{Tr}_{p}^{\ell p} \left(g_{p, \ell, j}(z) \mid_{\lambda} W_{\ell}^{p\ell}\right)$ does as well.
\noindent
We~set 
\begin{align} \begin{split}\label{eq:trace}
G_{p,\ell,j}(z) & = \ell^{\frac{\lambda}{2} - 1}\text{Tr}_{p}^{ \ell p} \left(g_{p,\ell,j}(z) \mid_{\lambda} W_{\ell}^{p \ell} \right) \\ 
& =  \ell^{\frac{\lambda}{2} - 1} g_{p,\ell,j}(z) \mid_{\lambda} W_{\ell}^{p \ell} + g_{p, \ell, j}(z)\mid U_{\ell}
\in M_{\lambda}(\Gamma_{0}(p)). 
\end{split}
\end{align}
Therefore, to conclude that $G_{p,\ell,j}(z) \equiv (f_{\ell,j}(z)f_{\ell,j}(pz))\mid U_{\ell^j} \Mod{\ell^j}$, we deduce from \eqref{g_cong} that it suffices to prove that   
\begin{gather} \label{GWO}
    \ell^{\frac{\lambda}{2} - 1} g_{p,\ell,j}(z) \mid_{\lambda} W_{\ell}^{p \ell}\equiv 0 \Mod{\ell^{j}},
\end{gather}
which we do in the following lemma.  The form $G_{p, \ell, j}(z)$ has rational coefficients; if \eqref{GWO} holds, then its coefficients are also $\ell$-integral.  If $G_{p,\ell,j}(z)$ does not have integer coefficients, then we use the fact that its coefficients are $\ell$-integral with bounded denominators to assert the existence of an integer $M\equiv 1 \pmod{\ell^j}$ such that $M\bigcdot G_{p, \ell, j}(z)$ has integer coefficients.  The form $M\bigcdot G_{p, \ell, j}(z)$ then satisfies the conclusion of Theorem \ref{T1}.

\begin{lemma} \label{key_lemma}
Let $\ell$ and $p$ be distinct primes with $\ell\geq 5$, and let $j\geq 1$.  Then we have
    \begin{gather*}
    \ell^{\frac{\lambda}{2} - 1}g_{p,\ell,j}(z) \mid_{\lambda} W_{\ell}^{p \ell}\equiv 0 \Mod{\ell^{j}}.
\end{gather*}
\end{lemma}
\begin{proof}
    Using \eqref{UVq}, \eqref{FVAL}, \eqref{hdef}, and \eqref{LP} we have
    \begin{align}
    \begin{split}
    \ell^{\frac{\lambda}{2} - 1} g_{p,\ell,j}(z) \mid_{\lambda} W_{\ell}^{p \ell} & =  \ell^{\frac{\lambda}{2} - 1}\left(f_{\ell,j}(z) f_{\ell,j}(pz)\right) \mid U_{\ell^{j-1}} \mid_{\ell^{j} - 1} W_{\ell}^{p\ell} \\ 
    & \bigcdot  h_{\ell,j}(z) \mid_{\frac{\ell^{j-1} (\ell - 1)}{2}} W_{\ell}^{p\ell} \bigcdot h_{\ell,j}(pz) \mid_{\frac{\ell^{j - 1}(\ell - 1)}{2}} W_{\ell}^{p\ell}. \label{prelim_comp}
\end{split}
\end{align}
\noindent    
We apply Proposition \ref{Commutativity} with $f(z) = h_{\ell,j}(z) \in M_{\frac{\ell^{j - 1}(\ell - 1)}{2}}\left(\Gamma_{0}(\ell), \left(  \frac{\bigcdot}{\ell} \right) \right)$ and $t = p$, we recall that $\lambda = \ell^j - 1 + \ell^{j - 1}(\ell - 1)$, and we use \eqref{hH} to continue the computation in \eqref{prelim_comp}:
    \begin{align}
    \begin{split}
         \ell^{\frac{\lambda}{2} - 1} & g_{p,\ell,j}(z) \mid_{\lambda} W_{\ell}^{p \ell} 
         =  \ell^{\frac{\lambda}{2} - 1}\left(f_{\ell,j}(z) f_{\ell,j}(pz)\right) \mid U_{\ell^{j-1}} \mid_{\ell^{j} - 1} W_{\ell}^{p\ell} \\
	& \bigcdot\left(\frac{p}{\ell}\right) h_{\ell, j}(z)\mid_{\frac{\ell^{j - 1}(\ell - 1)}{2}}W_{\ell^{p\ell}} 
	\bigcdot h_{\ell, j}(z)\mid_{\frac{\ell^{j - 1}(\ell - 1)}{2}} H_{\ell}\mid V_p \\
         & =  \ell^{\frac{\lambda}{2} - 1}\left(f_{\ell,j}(z) f_{\ell,j}(pz)\right) \mid U_{\ell^{j-1}} \mid_{\ell^{j} - 1} W_{\ell}^{p\ell} \\
         & \bigcdot\left( \frac{p}{\ell}\right) (-i)^{\frac{\ell^{j} - \ell^{j - 1}}{2}} \ell^{\frac{\ell^{j} + \ell^{j - 1}}{4}} f_{\ell, 1}(z)^{\ell^{j - 1}}
          \bigcdot(-i)^{\frac{\ell^{j} - \ell^{j - 1}}{2}} \ell^{\frac{\ell^{j} + \ell^{j - 1}}{4}} f_{\ell, 1}(pz)^{\ell^{j - 1}} \label{SWO}\\
          & = \left(\frac{-p}{\ell} \right)\ell^{\frac{3(\ell^{j} - 1)}{2}}\left(f_{\ell,j}(z) f_{\ell,j}(pz)\right) \mid U_{\ell^{j-1}} \mid_{\ell^{j} - 1} W_{\ell}^{p\ell} \bigcdot (f_{\ell, 1}(z) f_{\ell,1}(pz))^{\ell^{j - 1}}. 
          \end{split}
    \end{align}
We now use \eqref{SWO} to study the Fourier expansion of $\ell^{\frac{\lambda}{2} - 1}g_{p,\ell,j}(z) \mid_{\lambda} W_{\ell}^{p \ell}$ in the appropriate power series ring modulo $\ell^j$.  We define $\zeta_{D \ell^{j}} = e^{\frac{2 \pi i}{D \ell^{j}}}$ and $q_{\ell^{j}} = q^{\frac{1}{\ell^{j}}}$.  It follows from \eqref{SWO} that in $\mathbb{Z} [\zeta_{D \ell^{j}}][[q_{\ell^{j}}]]$, we have $\ell^{\frac{\lambda}{2} - 1} \left( g_{p,\ell,j}(z) \mid_{\lambda} W_{\ell}^{p \ell} \right)  \equiv 0 \Mod{\ell^{j}}$ if and only if 
\[
\ell^{\frac{3(\ell^{j} - 1)}{2}} (f_{\ell,j}(z) f_{\ell,j}(pz)) \mid U_{\ell^{j-1}} \mid_{\ell^{j} - 1} W_{\ell}^{p\ell} \equiv 0 \Mod{\ell^{j}},
\]
which holds if and only if 
\begin{equation} \label{val_cond}
    v_{\ell} \left((f_{\ell,j}(z) f_{\ell,j}(pz)) \mid U_{\ell^{j-1}} \mid_{\ell^{j} - 1} W_{\ell}^{p\ell} \right) \geq j - \frac{3(\ell^{j} - 1)}{2},
\end{equation} 
where $v_{\ell}(F)$ is the largest power of $\ell$ dividing every coefficient of $F\in \mathbb{Z} [\zeta_{D \ell^{j}}][[q_{\ell^{j}}]]$.

We let $x$, $y\in \Z$ with $y\ell^2 - p\ell x = \ell$.  From \eqref{UOP}, \eqref{ATI}, and \eqref{FVAL}, we observe that
    \begin{align} \begin{split}\label{UWD}
        (f_{\ell,j}(z) f_{\ell,j}(pz)) & \mid U_{\ell^{j-1}} \mid_{\ell^{j} - 1} W_{\ell}^{p\ell} \\
        & = \ell^{(j - 1) \left( \frac{\ell^{j} - 3}{2}  \right)} \mathlarger{\mathlarger{\sum}}_{m = 0}^{\ell^{j-1} - 1} \left(f_{\ell, j}(z)f_{\ell, j}(pz)   \right) \Big|_{\ell^{j} - 1} \begin{pmatrix}
            1 & m \\ 0 & \ell^{j - 1}
        \end{pmatrix} \begin{pmatrix}
            \ell & x \\ p \ell & y \ell
        \end{pmatrix}.
    \end{split}
    \end{align}
For $0\leq m \leq \ell^{j - 1} - 1$, we define $0\leq r\leq j - 2$ and $t\in\Z$ with $\ell\nmid t$ such that $1 + pm = \ell^{r} t$, and we define $0\leq b'\leq \ell^{j - r - 1} - 1$ such that $b^{'} \equiv t^{-1} (x + y m \ell) \Mod{\ell^{j - r - 1}}$.  We also require
\begin{gather*}
 \delta^{'} = y \ell^{r + 1} - p b^{'}, ~\beta^{'} = \frac{x + y m \ell - t b^{'}}{\ell^{j - r - 1}}.   
\end{gather*}
It follows that $\begin{pmatrix} t & \beta' \\ p\ell^{j - r - 1} & \delta'\end{pmatrix} \in \text{SL}_2(\Z)$ and that
\begin{gather} \label{CHID}
    \begin{pmatrix}
        1 & m \\ 0 & \ell^{j-1} \end{pmatrix} \begin{pmatrix}
            1 & x \\ p & y \ell
        \end{pmatrix} = \begin{pmatrix}
            t & \beta^{'} \\ p \ell^{j - r - 1} & \delta^{'} 
        \end{pmatrix} \begin{pmatrix}
            \ell^{r} & b^{'} \\ 0 & \ell^{j - r - 1}. 
        \end{pmatrix}
        \end{gather}
For all $0 \leq m \leq \ell^{j-1} - 1$, we compute 
    \begin{align}
    \begin{split}
    \begin{pmatrix} 
        1 & m \\ 0 & \ell^{j-1} 
    \end{pmatrix} \begin{pmatrix}
        \ell & x \\ p \ell & y \ell
    \end{pmatrix} & = \begin{pmatrix}
        1 & m \\ 0 & \ell^{j-1} \end{pmatrix}\begin{pmatrix}
            1 & x \\ p & y \ell
        \end{pmatrix} \begin{pmatrix}
            \ell & 0 \\ 0 & 1
        \end{pmatrix} \\
        & =\begin{pmatrix}
            t & \beta^{'} \\ p \ell^{j - r - 1} & \delta^{'} 
        \end{pmatrix} \begin{pmatrix}
            \ell^{r} & b^{'} \\ 0 & \ell^{j - r - 1} 
        \end{pmatrix} \begin{pmatrix}
            \ell & 0 \\ 0 & 1
        \end{pmatrix} \\
  & = \begin{pmatrix}
            t & \beta^{'} \\ p \ell^{j - r - 1} & \delta^{'} 
        \end{pmatrix} \begin{pmatrix}
            \ell^{r + 1} & b^{'} \\ 0 & \ell^{j - r - 1}
        \end{pmatrix} \label{COP},
        \end{split}
        \end{align} 
where we used \eqref{CHID} for the second equality.  Using \eqref{SLO}, \eqref{ETA}, and \eqref{COP}, we deduce that
        \begin{align} 
        \eta(z) \Big|_{1/2} \begin{pmatrix}
            1 & m \\ 0 & \ell^{j - 1}
        \end{pmatrix} \begin{pmatrix}
            \ell & x \\ p \ell & y \ell
        \end{pmatrix} & = \epsilon_{t, \beta^{'}, p \ell^{j - r - 1},  \delta^{'}} \, \bigcdot \, \ell^{\frac{r + 1}{2} - \frac{j}{4}}\, \bigcdot \, \eta\left( \frac{\ell^{r + 1}z + b^{'}}{\ell^{j - r - 1}} \right), \label{TC1} \\
         \eta(p z) \Big|_{1/2} \begin{pmatrix}
            1 & m \\ 0 & \ell^{j - 1}
        \end{pmatrix} \begin{pmatrix}
            \ell & x \\ p \ell & y \ell
        \end{pmatrix} & = \epsilon_{t, p \beta^{'}, \ell^{j - r - 1}, \delta^{'}}\, \bigcdot \, \ell^{\frac{r + 1}{2} - \frac{j}{4}}\, \bigcdot \, \eta \left( \frac{\ell^{r + 1} p z + p b^{'}}{\ell^{j - r - 1}} \right), \label{TC2} \\
        \eta(\ell^{j} z) \Big|_{1/2} \begin{pmatrix}
            1 & m \\ 0 & \ell^{j - 1}
        \end{pmatrix} \begin{pmatrix}
            \ell & x \\ p \ell & y \ell
        \end{pmatrix} & = \epsilon_{\ell (1 + pm) , x + y m \ell, p , y}\, \bigcdot \, \ell^{- \frac{j}{4}}\, \bigcdot \, \eta(z), \label{TC3} \\
        \eta(p \ell^{j} z) \Big|_{1/2} \begin{pmatrix}
            1 & m \\ 0 & \ell^{j - 1}
        \end{pmatrix} \begin{pmatrix}
            \ell & x \\ p \ell & y \ell
        \end{pmatrix} & = \epsilon_{\ell (1 + pm) , p (x + y m \ell), 1 , y}\, \bigcdot \, \ell^{- \frac{j}{4}}\, \bigcdot \, \eta(p z). \label{TC4}
        \end{align}
On substituting \eqref{TC1}, \eqref{TC2}, \eqref{TC3}, and \eqref{TC4} in \eqref{UWD}, we find, for $0\leq m\leq \ell^{j-1} - 1$, that the $m$th summand in the expression \eqref{UWD} for $(f_{\ell,j}(z) f_{\ell,j}(pz)) \mid U_{\ell^{j - 1}} \mid_{\ell^{j} - 1} W_{\ell}^{p\ell}$ is
\begin{align} \label{SVI}
 \ell^{(j - 1) \left( \frac{\ell^{j} - 3}{2}  \right)} \frac{\left( \ell^{-\frac{j}{4}} ~ \epsilon_{\ell (1 + pm),x + y m \ell,p,y} ~ \eta(z) \right)^{\ell^{j}} \left( \ell^{- \frac{j}{4}} ~ \epsilon_{\ell (1 + pm),p (x + y m \ell),1,y} ~ \eta(p z) \right)^{\ell^{j}}}{\left( \ell^{\frac{r + 1}{2} - \frac{j}{4}} ~ \epsilon_{t, \beta^{'},p \ell^{j - r - 1},\delta^{'}} ~ \eta\left( \frac{\ell^{r + 1}z + b^{'}}{\ell^{j - r - 1}}\right) \right) \left(\ell^{\frac{r + 1}{2} - \frac{j}{4}} ~ \epsilon_{t,p \beta^{'}, \ell^{j - r - 1},\delta^{'}} ~\eta \left( \frac{\ell^{r + 1} p z + p b^{'}}{\ell^{j - r - 1}} \right) \right)}. 
\end{align}
With
\begin{align*}
    \lambda_{\ell, p, m , j} = \frac{\epsilon_{\ell (1 + pm),x + y m \ell,p,y} \, \bigcdot \, \epsilon_{\ell (1 + pm),p (x + y m \ell),1,y}}{\epsilon_{t, \beta^{'},p \ell^{j - r - 1},\delta^{'}}\,\bigcdot\, \epsilon_{t,p \beta^{'}, \ell^{j - r - 1},\delta^{'}}} 
\end{align*}
in the group of $24$th roots of unity, \eqref{SVI} simplifies to 
\begin{align*} 
\ell^{- \left(\frac{\ell^{j} - 1}{2} + j + r\right)} ~~ \lambda_{\ell,p,m,j}~ ~\frac{\eta(z)^{\ell^{j}} \eta(pz)^{\ell^{j}}}{\eta\left( \frac{\ell^{2(r + 1)} z ~ + ~ \ell^{r + 1} b^{'}}{\ell^{j}}  \right)   \eta\left( \frac{\ell^{2(r + 1)} p z ~+~ p \ell^{r + 1} b^{'}}{\ell^{j}}  \right)}.
\end{align*}
From \eqref{UWD} and the definitions $\Delta = \gcd(24, p+1)$ and $D = \frac{24}{\Delta}$, it follows that $(f_{\ell, j}(z)f_{\ell, j}(pz))\mid U_{\ell^{j - 1}}\mid _{\ell^{j} - 1} W^{p\ell}_{\ell}$ has expansion in the power series ring $\mathbb{Z} [\zeta_{D \ell^{j}}][[q_{\ell^{j}}]]$ given by
\begin{align*}
\sum_{m = 0}^{\ell^{j - 1} - 1}\ell^{- \left(\frac{\ell^{j} - 1}{2} + j + r\right)} & \lambda_{\ell,p,m,j} ~ q_{\ell^{j}}^{\frac{\left(\ell^{2j} - \ell^{2(r + 1)}\right)\left(\frac{p + 1}{\Delta}\right)}{D}} ~ \zeta_{D \ell^{j}}^{-\ell^{r + 1}b'} \\ 
& \bigcdot\prod_{n = 1}^{\infty} \frac{(1 - q^{n})^{\ell^{j}} (1 - q^{pn})^{\ell^{j}}}{\left( 1 - \zeta_{D \ell^{j}}^{n \ell^{r + 1} b^{'}D} \bigcdot q_{\ell^{j}}^{n \ell^{2(r + 1)}} \right) \left( 1 - \zeta_{D \ell^{j}}^{n p \ell^{r + 1} b^{'}D} \bigcdot q_{\ell^{j}}^{n \ell^{2(r + 1)} p} \right)}.
\end{align*}
Therefore, we have
\begin{align*}
   v_{\ell} \left( f_{\ell,j}(z) f_{\ell,j}(pz) \mid U_{\ell^{j - 1}} \mid_{\ell^{j} - 1} W_{\ell}^{p\ell} \right) &\geq \text{min} \Bigl \{ - \left( \frac{\ell^{j} - 1}{2} + j + r \right) :  ~0 \leq r \leq j - 2 \Bigr \} \\
   & = - \left( \frac{\ell^{j} - 1}{2} \right) - 2j + 2. 
\end{align*}    
\noindent 
Since $~\ell \geq 5 ~\text{and}~ j \geq 1$, we have $- \left( \frac{\ell^{j} - 1}{2} \right) - 2j + 2 \geq j - 3 \left( \frac{\ell^{j} - 1}{2} \right)$, which proves Lemma \ref{key_lemma}, and with it, Theorem \ref{T1}.
\end{proof}

We turn to Theorem \ref{ord_vanish_G} on the orders of $G_{p, \ell, j}(z)$ as in \eqref{eq:trace} at cusps $\left[\frac{1}{p}\right]_{p} = [\infty]_{p}$ and $[1]_p = [0]_p$
of $\Gamma_0(p)$.  
\begin{theorem}\label{ord_vanish_G}
Let $p \neq \ell$ be primes with $\ell \geq 5$, and let $j \geq 1$. Let $G_{p,\ell,j}(z)$ be as in \eqref{eq:trace}, and let~$\delta$ be as in \eqref{B2}. Then we have
\begin{enumerate}
    \item $\operatorname{ord}_{\left[\frac{1}{p}\right]_{p}}(G_{p,\ell,j}(z)) \geq \lceil\delta/\ell^j \rceil$ and
    \item $\operatorname{ord}_{[1]_{p}}(G_{p,\ell,j}(z)) \geq \lceil\delta/\ell^j\rceil$.
\end{enumerate}
\end{theorem}

We first prove part (1) of Theorem \ref{ord_vanish_G}. With $f_{\ell, j}(z)$ as in \eqref{FVAL}, $h_{\ell, j}(z)$ as in \eqref{hdef}, and $g_{p, \ell, j}(z)$ as in \eqref{LP}, we define
\begin{equation}
    A(z) =  \left(f_{\ell,j}(z) f_{\ell,j}(pz)\right) \mid U_{\ell^{j-1}} \mid_{\ell^{j} - 1} W_{\ell}^{p\ell} \bigcdot  (f_{\ell, 1}(z) f_{\ell,1}(pz))^{\ell^{j - 1}}, \label{Adef}
\end{equation}
\begin{equation}
    B(z) = g_{p,\ell, j}(z)\mid U_{\ell} = [(f_{\ell, j}(z) f_{\ell,j}(pz)) \mid U_{\ell^{j-1}} \bigcdot   h_{\ell,j}(z) h_{\ell, j}(pz) ] \mid U_{\ell}, \label{Bdef}
\end{equation}
\begin{equation}
    C(z) = \left(f_{\ell,j}(z) f_{\ell,j}(pz)\right) \mid U_{\ell^{j-1}} |_{\ell^{j} - 1} W_{\ell}^{p \ell}. \label{Cdef}
\end{equation}
It follows from \eqref{umap}, the comments following \eqref{ATI}, \eqref{FVAL}, and \eqref{hdef} that $A(z)$, $B(z)$, and $C(z)$ are holomorphic modular forms on $\Gamma_0(p\ell)$.  
From \eqref{LP}, \eqref{eq:trace}, and \eqref{SWO}, we obtain
\begin{align}\label{eq:ord_AB}
    G_{p,\ell,j}(z) = \ell^{\frac{\lambda}{2} - 1}  ~ \text{Tr}_{p}^{ \ell p} \left(g_{p,\ell,j}(z) \mid_{\lambda} W_{\ell}^{p \ell} \right) 
    = \ell^{\frac{3(\ell^{j} - 1)} {2}} \left(\frac{-p}{\ell} \right) A(z) + B(z).
\end{align}
Using \eqref{FVAL}, \eqref{Adef}, \eqref{Cdef}, the holomorphy of $C(z)$, and Proposition \ref{CE}, we find that 
\begin{align}\label{eq:ord_A}
    \operatorname{ord}_{\left[\frac{1}{p \ell}\right]_{p \ell}}\left( A(z) \right) 
    = \text{ord}_{\left[\frac{1}{p \ell}\right]_{p \ell}}(C(z)) 
    + \text{ord}_{\left[\frac{1}{p\ell}\right]_{p\ell}}(f_{\ell, 1}(z) f_{\ell,1}(pz))^{\ell^{j - 1}}) 
    \geq \ell^{j - 1}\bigcdot\frac{(\ell^2 - 1)(p + 1)}{24},
    \end{align}
    where $\left[\frac{1}{p\ell}\right]_{p\ell} = [\infty]_{p\ell}$. 
Using \eqref{hdef}, Proposition \ref{CE} yields $\operatorname{ord}_{\left[\frac{1}{p \ell}\right]_{p \ell}}(h_{\ell,j}(z) h_{\ell,j}(pz)) = 0$.  Therefore, \eqref{UVq}, \eqref{FVAL}, \eqref{Bdef}, and Proposition \ref{CE} imply that 
\begin{equation}\label{eq:ord_B}
    \operatorname{ord}_{\left[\frac{1}{p \ell}\right]_{p \ell}}(B(z)) 
    = \operatorname{ord}_{\left[\frac{1}{p \ell}\right]_{p \ell}} ((f_{\ell, j}(z) f_{\ell,j}(pz))\mid U_{\ell^{j}}) 
    \geq \left \lceil \frac{\delta}{\ell^j} \right\rceil. 
\end{equation} 
To conclude the first part of the theorem, we recall that $\delta = \frac{(p + 1)(\ell^{2j} - 1)}{24}$ as in \eqref{B2}, we note that 
$\left[\frac{1}{p\ell}\right]_p = \left[\frac{1}{p}\right]_p$, and we use \eqref{eq:ord_AB}, \eqref{eq:ord_A}, and \eqref{eq:ord_B} to see that 
\begin{align*}
    \text{ord}_{\left[\frac{1}{p}\right]_p}(G_{p, \ell, j}(z)) & = \text{ord}_{\left[\frac{1}{p\ell}\right]_p}(G_{p, \ell, j}(z))
     = \text{ord}_{\left[\frac{1}{p\ell}\right]_{p\ell}}(G_{p, \ell, j}(z)) \\
    &\geq \min \left\{ \operatorname{ord}_{\left[\frac{1}{p \ell}\right]_{p \ell}}(A(z)), \: \operatorname{ord}_{\left[\frac{1}{p \ell}\right]_{p \ell}}(B(z))   \right\} \\
    & \geq\text{min}\left\{\ell^{j - 1}\bigcdot\frac{(\ell^2 - 1)(p + 1)}{24}, \, \left\lceil\frac{\delta}{\ell^j}\right\rceil\right\}= \left \lceil \frac{\delta}{\ell^j} \right\rceil,
\end{align*}
where the second equality holds since the width of the cusp is the same in levels $p\ell$ and $p$.  

\medskip

Before proceeding to the proof of part (2) of Theorem \ref{ord_vanish_G}, we record further useful facts.
\begin{proposition}\label{slash_Wp}
Let $k, j \in \Z$ with $j\geq 1$ and let $\ell\neq p$ be primes.
\begin{enumerate}
\item Suppose that $F(z)\in M_k(\Gamma_0(p\ell))$.  We have $\operatorname{ord}_{\left[\frac{1}{\ell}\right]_{p\ell}}(F(z)) = \operatorname{ord}_{\left[\frac{1}{p \ell}\right]_{p \ell}}(F(z) \mid_{k} W_{p}^{\ell p})$.
\item Suppose that $G(z)\in M_k(\Gamma_0(\ell^jp))$ and that $t\geq j$.  We have 
\begin{enumerate}
    \item $G(z) \mid_{k} W_{p}^{\ell^{t}p} = G(z) \mid_{k} W_{p}^{\ell^{j}p}$ and
    \item $G(z) \mid U_{\ell^{j - 1}}\mid_k W_{p}^{\ell p} = G(z)\mid_k W_p^{\ell^jp}\mid U_{\ell^{j - 1}}$.
\end{enumerate} 
\item Suppose that $\psi$ is a character modulo $\ell^j$ and that $H(z)\in M_k(\Gamma_0(\ell^j),\psi)$.  Then we have
\begin{equation*}
(H(z)H(pz))|_{2k}W_p^{\ell^jp} = \overline{\psi}(p)H(z)H(pz).
\end{equation*}
\end{enumerate}
\end{proposition}
\begin{proof}
    For part (1), we let $x$, $y \in \Z$ with $py - \ell x = 1$, and we let 
    \begin{equation}
        W_{p}^{\ell p} = \begin{pmatrix}p & x \\ p\ell & py\end{pmatrix} = \begin{pmatrix} 1 & x \\ \ell & py\end{pmatrix}\begin{pmatrix}p & 0 \\ 0 & 1\end{pmatrix}. \label{W_fact}
    \end{equation}
    We note that the cusps $\left[\frac{1}{\ell}\right]_{p\ell}$ and $\left[\frac{1}{p\ell}\right]_{p\ell}$ have widths $p$ and $1$, respectively, and we observe that $F(z)$ has expansion at $\left[\frac{1}{\ell}\right]_{p\ell}$ in the local variable $q_p$ given by $F(z)\mid_k \begin{pmatrix} 1 & x \\ \ell & py\end{pmatrix} = \sum c(n)q_p^n$.  Using \eqref{W_fact}, we~compute
    \begin{align*}
        F(z)\mid_k W_{p}^{\ell p} = F(z)\mid_k \begin{pmatrix} 1 & x \\ \ell & py\end{pmatrix}\begin{pmatrix} p & 0 \\ 0 & 1\end{pmatrix}
        = \sum c(n)q_p^n\mid_k\begin{pmatrix} p & 0 \\ 0 & 1\end{pmatrix} = p^{k/2}\sum c(n)q^n,
    \end{align*}
    which proves part (1) of the proposition.

    \medskip

    For part (2), we let $x$, $y$, $x'$, $y'\in \Z$ with $py - \ell^t x = py' - \ell^jx' = 1$, and we let 
    \begin{equation}
        W_{p}^{\ell^jp} = \begin{pmatrix} p & x \\ \ell^j p & py \end{pmatrix}, \
        W_{p}^{\ell^tp} = \begin{pmatrix} p & x' \\ \ell^t p & py'\end{pmatrix}. \label{W_power}
    \end{equation}
    We compute $M = W_p^{\ell^tp}\bigcdot \left(W_{p}^{\ell^{j}p}\right)^{-1} = \begin{pmatrix} py - \ell^jx' & -x + x' \\ p\ell^j(p^{t - j}y - y') & -\ell^tx + py'\end{pmatrix} \in \Gamma_0(\ell^j p)$.  It follows that $G(z)\mid_k W_p^{\ell^tp} = G(z)\mid_k MW_p^{\ell^jp} = G(z)\mid_k W_p^{\ell^jp}$, which is (2a).  For (2b), we note that $G(z)\mid_k W_p^{\ell^jp}\mid U_{\ell^{j - 1}} = G(z)\mid U_{\ell^{j -1}}\mid_k W_p^{\ell^jp} = G(z)\mid U_{\ell^{j - 1}}\mid_k W_p^{\ell p}$, where the first equality results from $j - 1$ applications of Proposition \ref{atkin-li} and the second results from part (2a).    

    \medskip

    We next observe that $W_{p}^{\ell^j p}$ as in \eqref{W_power} has  
    \begin{equation*}
        W_{p}^{\ell^jp} = \begin{pmatrix} p & x \\ \ell^j p & py\end{pmatrix} = \begin{pmatrix} 1 & x \\ \ell^j & py\end{pmatrix} \begin{pmatrix} p & 0 \\ 0 & 1\end{pmatrix},
    \end{equation*}
    with $\begin{pmatrix} 1 & x \\ \ell^j & py\end{pmatrix} \in \Gamma_0(\ell^j)$ and that $py\equiv 1\pmod{\ell^j}$.  We then have 
    \begin{equation}
        H(z)\mid_k W_p^{\ell^j p} = H(z)\mid_k \begin{pmatrix} 1 & x \\ \ell^j & py\end{pmatrix}
        \begin{pmatrix} p & 0 \\ 0 & 1\end{pmatrix} = \psi(py) H(z)\mid_k \begin{pmatrix} p & 0 \\ 0 & 1\end{pmatrix} = p^{k/2}H(pz).\label{H}
    \end{equation}
    We also observe that $\begin{pmatrix} p & 0 \\ 0 & 1 \end{pmatrix}W_{p}^{\ell^j p} = \begin{pmatrix} p & 0 \\ 0 & p\end{pmatrix}\begin{pmatrix} p & x \\ \ell^j & y\end{pmatrix}$ with $\begin{pmatrix} p & x \\ \ell^j & y\end{pmatrix} \in \Gamma_0(\ell^j)$ and that $\psi(y) = \overline{\psi}(p)$.  These facts imply that
    \begin{align}
        \begin{split}
        H(pz)\mid_k W_p^{\ell^j p} & = p^{-k/2}H(z)\mid_k\begin{pmatrix} p & 0 \\ 0 & 1\end{pmatrix}W_p^{\ell^j p} \\ & = p^{-k/2}H(z)\mid_k \begin{pmatrix} p & x \\ \ell^j & y\end{pmatrix} = p^{-k/2}\overline{\psi}(p)H(z). \label{Hp}
        \end{split}
    \end{align}
    We now use \eqref{H} and \eqref{Hp} to conclude that 
    \begin{equation*}
        (H(z)H(pz))\mid_{2k} W_{p}^{\ell^jp} = H(z)\mid_k W_p^{\ell^j p}\bigcdot H(pz)\mid_k W_p^{\ell^j p}
        = \overline{\psi}(p)H(z)H(pz),
    \end{equation*}
    which is part (3) of the proposition.
\end{proof}

We recall that $A(z)$ as in \eqref{Adef} and $B(z)$ as in \eqref{Bdef} lie in $M_{\lambda}(\Gamma_0(p\ell))$ with $\lambda$ as in \eqref{eq:lambda_def}.  To obtain lower bounds on $\text{ord}_{\left[\frac{1}{\ell}\right]_{p\ell}}(A(z))$ and $\text{ord}_{\left[\frac{1}{\ell}\right]_{p\ell}}(B(z))$, we consider $A(z)\mid_{\lambda} W_{p}^{\ell p}$ and $B(z)\mid_{\lambda} W_p^{\ell p}$. 
\noindent
Using \eqref{Adef} we find that 
\begin{align}
A(z)\mid_{\lambda} W_{p}^{\ell p} & = (f_{\ell, j}(z)f_{\ell,j}(pz))\mid U_{\ell^{j - 1}}\mid_{\ell^{j} - 1} H_{\ell p}\bigcdot (f_{\ell, 1}(z)f_{\ell, 1}(pz))^{\ell^{j - 1}}\mid_{\ell^{j - 1}(\ell - 1)}W_p^{\ell p} \notag \\ 
& = \left(\frac{p}{\ell}\right)(f_{\ell, j}(z)f_{\ell,j}(pz))\mid U_{\ell^{j - 1}}\mid_{\ell^{j} - 1} H_{\ell p}\bigcdot(f_{\ell, 1}(z)f_{\ell, 1}(pz))^{\ell^{j - 1}}, \label{A_simple}
\end{align}
where the first equality follows since $W_{\ell}^{\ell p}\circ W_{p}^{\ell p}$ and $H_{\ell p}$ as in \eqref{Fricke} are $\Gamma_0(p\ell)$-equivalent, and the second equality follows from the third part of Proposition \ref{slash_Wp}.  We note that $(f_{\ell, j}(z)f_{\ell, j}(pz))\mid U_{\ell^{j - 1}}\mid H_{\ell p}$ has non-negative order at $\left[\frac{1}{\ell}\right]_{p\ell}$ since it is a holomorphic modular form.  Therefore, using \eqref{Adef}, \eqref{A_simple}, Proposition \ref{CE}, and the first part of Proposition \ref{slash_Wp}, we compute
\begin{align}
\text{ord}_{\left[\frac{1}{\ell}\right]_{p\ell}}(A(z)) & = \text{ord}_{\left[\frac{1}{p\ell}\right]_{p\ell}}(A(z)|_{\lambda}W_p^{\ell p}) \notag \\ 
& = \text{ord}_{\left[\frac{1}{p\ell}\right]_{p\ell}}((f_{\ell,j}(z)f_{\ell,j}(pz))\mid U_{\ell^{j - 1}}\mid_{\ell^j - 1} H_{\ell p}) + \text{ord}_{\left[\frac{1}{p\ell}\right]_{p\ell}}((f_{\ell, 1}(z)f_{\ell, 1}(pz))^{\ell^{j - 1}}) \notag \\
& \geq \text{ord}_{\left[\frac{1}{p\ell}\right]_{p\ell}}((f_{\ell, 1}(z)f_{\ell, 1}(pz))^{\ell^{j - 1}})
= \ell^{j - 1}\bigcdot\frac{(\ell^{2} - 1)(p + 1)}{24}. \label{ord_A}
\end{align}

It remains to study $\operatorname{ord}_{\left[\frac{1}{\ell}\right]_{p \ell}}(B(z))$. Since $\ell\neq p$ and $g_{p, \ell, j}(z)$ has level~$p\ell$, \eqref{Bdef} and the proof of Proposition 7 of \cite{martin} imply that 
\begin{equation}
\operatorname{ord}_{\left[\frac{1}{\ell}\right]_{p \ell}}(B(z)) = \text{ord}_{\left[\frac{1}{\ell}\right]_{p\ell}}(g_{p, \ell, j}(z)\mid U_{\ell}) \geq 
\left\lceil\frac{1}{\ell}\bigcdot\text{ord}_{\left[\frac{1}{\ell}\right]_{p\ell}}(g_{p, \ell, j}(z))\right\rceil. \label{B_ord}
\end{equation}
Using \eqref{LP}, we obtain
\begin{align}\begin{split}\label{g_simple}
g_{p, \ell, j}(z)\mid_{\lambda}W_{p}^{\ell p}
& = (f_{\ell, j}(z)f_{\ell, j}(pz))\mid U_{\ell^{j - 1}}\mid_{\ell^{j} - 1}W_p^{\ell p}\bigcdot(h_{\ell, j}(z)h_{\ell, j}(pz))\mid_{\ell^{j - 1}(\ell - 1)}W_p^{\ell p}\\
& = (f_{\ell, j}(z)f_{\ell, j}(pz))\mid_{\ell^j - 1}W_p^{\ell^jp}\mid U_{\ell^{j - 1}}
\bigcdot (h_{\ell, j}(z)h_{\ell, j}(pz))\mid_{\ell^{j - 1}(\ell - 1)}W_p^{\ell p} \\
& = \left(\frac{p}{\ell}\right)^{j + 1}(f_{\ell, j}(z)f_{\ell,j}(pz))\mid U_{\ell^{j - 1}}\bigcdot h_{\ell, j}(z)h_{\ell, j}(pz) \\
& = \left(\frac{p}{\ell}\right)^{j + 1}g_{p, \ell, j}(z),
\end{split}
\end{align}
where the second equality requires Proposition \ref{atkin-li} and the second part of Proposition \ref{slash_Wp}, and the third equality uses two applications of the third part of Proposition \ref{slash_Wp}.  We use \eqref{UVq}, \eqref{g_simple}, and the first part of Proposition \ref{slash_Wp} to compute
\begin{align}
    \text{ord}_{\left[\frac{1}{\ell}\right]_{p\ell}}(g_{p, \ell, j}(z))
    & =\text{ord}_{\left[\frac{1}{p\ell}\right]_{p\ell}}(g_{p, \ell, j}(z) \mid_{\lambda} W_p^{\ell p}) \notag \\
    & = \text{ord}_{\left[\frac{1}{p\ell}\right]_{p\ell}}((f_{\ell,j}(z)f_{\ell,j}(pz))\mid U_{\ell})
    + \text{ord}_{\left[\frac{1}{p\ell}\right]_{p\ell}}(h_{\ell, j}(z)h_{\ell, j}(pz)) \notag \\
    &\geq \left\lceil\frac{1}{\ell^{j - 1}}\bigcdot\frac{(\ell^{2j} - 1)(p + 1)}{24}\right\rceil = \left\lceil\frac{\delta}{\ell^{j - 1}}\right\rceil. \label{g_ord}
    \end{align}
Combining \eqref{B_ord} and \eqref{g_ord} yields
\begin{equation}
   \operatorname{ord}_{\left[\frac{1}{\ell}\right]_{p \ell}}(B(z)) 
   \geq \left\lceil\frac{1}{\ell}\left\lceil\frac{\delta}{\ell^{j - 1}}\right\rceil\right\rceil
   = \left\lceil\frac{\delta}{\ell^j}\right\rceil.  \label{ord_B_ell}
\end{equation}
To conclude the second part of the theorem, we observe that $[1]_p = \left[\frac{1}{\ell}\right]_p$ and we use \eqref{eq:ord_AB}, \eqref{ord_A}, and \eqref{ord_B_ell} to obtain
\begin{align*}
    \text{ord}_{[1]_p}(G_{p, \ell, j}(z)) 
    & = \text{ord}_{\left[\frac{1}{\ell}\right]_{p}}(G_{p, \ell, j}(z))
    = \text{ord}_{\left[\frac{1}{\ell}\right]_{p\ell}}(G_{p, \ell, j}(z)) \\
     & \geq \text{min}
     \left\{
     \text{ord}_{\left[\frac{1}{\ell}\right]_{p\ell}}(A(z)), \, 
    \text{ord}_{\left[\frac{1}{\ell}\right]_{p\ell}}(B(z))
    \right\} \\
    & \geq\text{min}\left\{\ell^{j - 1}\bigcdot\frac{(\ell^2 - 1)(p + 1)}{24}, \left\lceil\frac{\delta}{\ell^j}\right\rceil\right\}= \left \lceil \frac{\delta}{\ell^j} \right\rceil,
\end{align*}
where the second equality holds since the width of the cusp $1/\ell$ is the same in levels $p\ell$ and $\ell$.

\subsection{Conclusion of the proof of Theorem \ref{MT1}}
From \eqref{U_on_f} and Theorem \ref{T1}, we deduce that there exists $G_{p, \ell, j}(z)\in M_{\lambda}(\Gamma_0(p))\cap \Z[[q]]$ such that 
\begin{equation}\label{eq:G}
   \sum_{m = 0}^{\infty} p_{[1,p]}(\ell^j m - \delta)q^
{m - \frac{\ell^j\left(\frac{p + 1}{\Delta}\right)}{D}} \equiv \frac{G_{p, \ell, j}(z)}{(\eta(z) \eta(pz))^{\ell^j}} \Mod{\ell^j}.
\end{equation}
Recalling the definition of $t$ in \eqref{defn_t}, Proposition \ref{CE} and Theorem \ref{ord_vanish_G} imply for all cusps $s$ that 
\begin{equation*}
    \operatorname{ord}_{s}\left( (\eta(z) \eta(pz))^{Dt} \right) = Dt \left(  \frac{p + 1}{24} \right) = t \left( \frac{p + 1}{\Delta} \right) \leq \Big\lceil \frac{\delta}{\ell^j} \Big\rceil \leq \operatorname{ord}_{s}(G_{p, \ell, j}(z)).
\end{equation*}
Hence, 
\begin{equation}
\label{H_def}
H_{p, \ell, j}(z) = \frac{G_{p, \ell, j}(z)}{(\eta(z)\eta(pz))^{Dt}}
\end{equation}
is holomorphic at the cusps and lies in  $M_{\lambda - Dt} \left( \Gamma_{0}(p), \left( \frac{-p}{\bigcdot} \right)^{Dt} \right)$.  Theorem \ref{MT1} now follows from \eqref{eq:G} and \eqref{H_def}: 
\begin{align*}
\sum_{m = 0}^{\infty} p_{[1,p]}(\ell^j m - \delta)q^
{m - \frac{\ell^j\left(\frac{p + 1}{\Delta}\right)}{D}} & \equiv \frac{G(z)}{(\eta(z) \eta(pz))^{\ell^j}} \equiv \frac{(\eta(z)\eta(pz))^{Dt} H(z)}{(\eta(z)\eta(pz))^{\ell^j}}  \\ 
&\equiv (\eta(z)\eta(pz))^{Dt - \ell^j} H(z) \Mod{\ell^j}.
\end{align*}

\noindent
\subsection{Proof of Theorem \ref{E5}}
For primes $p$, we recall that 
$D = \frac{24}{\gcd(p + 1, 24)}$.  For primes $\ell \neq p$ with $\ell\geq 5$ and $p + 1\mid 24$, and for $j\geq 1$, we recall that $y = Dt - \ell^j$, $k = \lambda - Dt$, $\chi = \left(\frac{-p}{\bigcdot}  \right)^{Dt}$ and $\mathcal{A}_{p, y, k, \chi}\subseteq M_{k + y}\left(\Gamma_0(pD^2), \left(\frac{-p}{\bigcdot}\right)\right)$ as in \eqref{Aspace}.  As in the remarks following Theorem~\ref{warnock_boylan}, we let $d = d_{p, \ell, j}$ be the dimension of $\mathcal{A}_{p, k, y, \chi}$, and we let $\{f_1,\ldots, f_{d}\}$ be a $\Z$-basis for the $\Z$-module $\mathcal{A}_{p, y, k, \chi}\cap \Z\llbracket q\rrbracket$.  We let $m\geq 5$ be prime with $m\not\in \{\ell, p\}$, and we let $\vec{f} = \langle f_1,\ldots, f_d\rangle^{\text{t}}$.  By Theorem \ref{warnock_boylan}, there exists $M \in \text{Mat}_{d \times d}(\Z)$ such that 
$\vec{f}\mid T_{m^2} = M\vec{f}$.  We also require 
\begin{equation*}
A =  \begin{pmatrix}
M - \left( \frac{-p}{m}  \right) m^{k + y - 1} I_{d} & - m^{2(k + y) - 2} I_{d} \\
I_{d} & 0_{d}, 
\end{pmatrix}
\end{equation*}
as in \eqref{A_matrix}.
In this notation, we use the following lemma to prove Theorem \ref{E5}. 

\begin{lemma} \label{E8}
Let $\ell$, $m$, and $p$ be distinct primes with $\ell$ and $m\geq 5$ and $p\in\{2, 3, 5\}$.  For all $i\geq 0$, we define $M_{i}, N_{i}$ and $O_{i}\in \textup{Mat}_{d\times d}(\Z)$ as follows: We have $M_0 = I_d$ and $N_0 = O_0 = 0_d$, and for all $i\geq 1$, we have 
\begin{equation} \label{E9}
\begin{pmatrix} M_{i} \\ M_{i - 1} 
\end{pmatrix} 
= A^{i} \begin{pmatrix} I_{d} \\ 0_{d} 
\end{pmatrix},
\ N_{i} = \left( \frac{-p}{m}\right) m^{k + y - 1} M_{i - 1}, \ \text{and} \ O_{i} = - m^{2(k + y) - 2} M_{i - 1}.
\end{equation}
Then for all $i\geq 1$, we have 
\begin{align} \label{E2}
\vec{f} \mid U_{m^{2}}^{i} = M_{i} \vec{f} + N_{i} (\vec{f} \otimes 1_{m}) + O_{i} (\vec{f} \mid V_{m^{2}}).
\end{align}
\end{lemma}

\begin{proof}
The proof proceeds by induction. We let $1_m$ denote the trivial character modulo $m$. From \eqref{UVq} and \eqref{twist}, we observe that $\vec{f} \mid U_{m} \mid V_{m} = \vec{f} - (\vec{f} \otimes 1_{m})$. Taking $n = m^{2}$ in \eqref{E4}, we deduce that
\begin{align}
\vec{f} \mid U_{m^{2}} & = \vec{f} \mid T_{m^{2}} - \left(   \frac{-p}{m}  \right) m^{k + y - 1} (\vec{f} - (\vec{f} \otimes 1_{m})) - m^{2(k + y) - 2}(\vec{f} \mid V_{m^{2}}) \notag \\
&= \left(M - \left(\frac{-p}{m}  \right)m^{k + y - 1} I_{d}\right)\vec{f} + \left(   \frac{-p}{m}  \right) m^{k + y - 1} (\vec{f} \otimes 1_{m})- ~ m^{2(k + y) - 2}  (\vec{f} \mid V_{m^{2}}) \notag \\
&= M_{1} \vec{f} + N_{1}  (\vec{f} \otimes 1_{m}) + O_{1} (\vec{f} \mid V_{m^{2}}), 
\label{E1}
\end{align}
where the second equality follows from $\vec{f}\mid T_{m^2} = M\vec{f}$.  We now let $i\geq 1$, and we suppose that \eqref{E2} holds.  
From \eqref{UVq} and \eqref{twist}, we note that $(\vec{f}\otimes 1_m)\mid U_{m^2} = \vec{0}$ and that $\vec{f}\mid V_{m^2}\mid U_{m^2} = \vec{f}$.  Applying $U_{m^2}$ to both sides of \eqref{E2} and using these facts, we obtain
\begin{align*}
\vec{f} \mid U_{m^{2}}^{i + 1} & = M_{i} (\vec{f} \mid U_{m^{2}}) + O_{i} \vec{f}
= (M_iM_1 + O_i)\vec{f} + M_iN_1(\vec{f}\otimes 1_m) + M_iO_1(\vec{f}\mid V_{m^2}) \\
& = M_{i + 1}\vec{f} + N_{i + 1}(\vec{f}\otimes 1_m) + O_{i + 1}(\vec{f}\mid V_{m^2}),
\end{align*}
where the second equality holds on substituting \eqref{E1} for $\vec{f} \mid U_{m^{2}}$, and \eqref{E9} precisely encodes the third equality.
\end{proof}
\noindent
We now turn to the proof of Theorem \ref{E5}. We suppose that $A\in \text{GL}_{2d \times 2d}(\Z)$ as in \eqref{A_matrix} has order $J$ in $\text{PGL}_{2d}(\Z/\ell^j\Z)$.  Then there exists $c \in (\mathbb{Z}/ \ell^{j}\mathbb{Z})^{\times}$ such that $A^{J} \equiv c I_{2d} \Mod{\ell^{j}}$.  We let $v \geq 1$.  With $i = vJ - 1$,  \eqref{E9} becomes 
\begin{align*}
\begin{pmatrix}
M_{v J - 1} \\ M_{v J - 2}
\end{pmatrix} 
& =  A^{v J - 1} 
\begin{pmatrix}
I_{d}\\0_{d}
\end{pmatrix} 
= A^{vJ}\bigcdot A^{-1}
\begin{pmatrix}
I_{d}\\
0_{d}
\end{pmatrix} 
\\
& \equiv c^{v} ~I_{2d} \, \bigcdot \, m^{-(2(k + y) - 2)}
\begin{pmatrix} 0_{d} & m^{2(k + y) - 2}I_{d} \\
- I_{d} & M - \left(\frac{-p}{m}\right) m^{k + y - 1} I_{d}
\end{pmatrix}\bigcdot
\begin{pmatrix}
I_{d} \\ 0_{d}
\end{pmatrix} \\
&\equiv -c^{v} m^{-(2(k + y) - 2)} 
\begin{pmatrix} 
0_{d} \\ 
I_{d}
\end{pmatrix} \Mod{\ell^j}.
\end{align*}

It follows that 
$M_{vJ - 1} \equiv 0_d \Mod{\ell^j}$ and $M_{vJ - 2} \equiv -c^vm^{-(2(k + y)-2)}I_d \Mod{\ell^j}$.  Lemma \ref{E8} now yields
	\begin{align*}
		\Vec{f} \mid U_{m^{2}}^{v J - 1}  
		\equiv  N_{v J  - 1} \left(\Vec{f} \otimes 1_{m}  \right) + O_{vJ - 1} \left( \Vec{f} \mid V_{m^{2}} \right) \Mod{\ell^{j}}.
	\end{align*}
Applying $U_{m}$ and noting for $\alpha \geq 1$ that $\displaystyle{U_{m^{2}}^{\alpha} = U_{m^{2 \alpha}}}$, we obtain	
\begin{align}
	\vec{f} \mid U_{m^{2vJ - 1}} & = \vec{f} \mid U_{m^{2}}^{v J - 1} \mid U_{m} \notag \\ 
    & = N_{v J - 1} \left(\vec{f} \otimes 1_{m} \right) \mid U_{m} + ~ O_{vJ - 1}\left(\vec{f} \mid V_{m^{2}} \mid U_{m}\right)
	= O_{v J - 1} \left( \vec{f} \mid V_{m} \right), \label{E12}
\end{align}
where the third equality results from $(\vec{f}\otimes 1_m)\mid U_m =\vec{0}$ and $\vec{f}\mid V_{m^2}\mid U_m = \vec{f}\mid V_m$ by \eqref{UVq} and \eqref{twist}.
From \eqref{E9}, we observe that 
\begin{equation*}
    O_{v J - 1} = M_{v J - 2} \bigcdot O_{1} = - c^{v} m^{- (2(k + y) - 2)} I_{d} \, \bigcdot \, -m^{(2(k + y) - 2)} I_{d} \equiv c^{v} I_{d} \Mod{\ell^{j}}. 
\end{equation*}
Substitution in \eqref{E12} gives
\begin{equation}
\vec{f} \mid U_{m^{2vJ - 1}}  \equiv c^{v} \vec{f} \mid V_{m} \Mod{\ell^{j}}. \label{vector_cong}
\end{equation}

To conclude the proof of Theorem \ref{E5}, we translate the vector congruence \eqref{vector_cong} into congruences for modular form coefficients.  For all $1\leq i\leq d$, we let $f_i(z) = \sum a_i(n)q^n$, and we recall that $f_i$ is the $i$th component function of $\vec{f}$.  We apply \eqref{UVq} to \eqref{vector_cong} to obtain
\[
\sum a_i(m^{2vJ-1}n)q^n\equiv c^v\sum a_i\left(\frac{n}{m}\right)q^n \Mod{\ell^j}.
\]
For all $1 \leq i \leq d$ and for all $n \geq 0$, comparing coefficients yields
\begin{gather}\label{MM}
a_{i}(m^{2vJ - 1}n) \equiv \begin{dcases} c^{v} a_{i_{1}} \left( \frac{n}{m}  \right) \Mod{\ell^{j}} \quad &\text{if} \ m \mid n, \\
    0 \Mod{\ell^{j}} \quad & \text{if} \ m \nmid n. 
\end{dcases}
\end{gather}
\noindent
The first remark following Theorem \ref{MT1} and \eqref{Aspace} imply, for primes $p$ with $p + 1\mid 24$, that 
$\sum p_{[1, p]}\left(\frac{\ell^jn + 1}{D}\right)q^n$ lies in $\mathcal{A}_{p, y, k, \chi} \cap \Z[[q]] \subseteq M_{s}\left(\Gamma_0(pD^2),\left(\frac{-p}{\bigcdot}\right)\right)$ modulo $\ell^j$.  Since the components of $\vec{f}$ are basis functions of the $\Z$-module $\mathcal{A}_{p, y, k, \chi}\cap \Z[[q]]$, there exists $\beta_{1}, \ldots, \beta_{d} \in \mathbb{Z}$ with 
\begin{gather} \label{lin_comb}
     \sum p_{[1, p]}\left( \frac{\ell^{j} n + 1}{D}    \right) q^{n} \equiv \sum_{i = 1}^{d} \beta_{i} f_{i} \equiv \sum  \left( \sum_{i = 1}^{d} \beta_{i} a_{i}(n) \right) q^{n} \Mod{\ell^{j}}. 
\end{gather}
Comparing coefficients with index $m^{2vJ - 1}n$ for $n$ with $m\nmid n$ and using \eqref{MM} gives the congruence \eqref{E13}, which is the first part of Theorem \ref{E5}.
  
We now prove \eqref{E14}, which is the second part of Theorem \ref{E5}. We let $O(A) = N$ in $\textup{GL}_{2d} (\mathbb{Z}/ \ell^{j} \mathbb{Z})$, and we apply \eqref{E9} from Lemma \ref{E8} with $i = N$ to obtain
\begin{gather*}
    \begin{pmatrix}
        M_{N} \\ M_{N - 1} \end{pmatrix} \equiv A^{N} \begin{pmatrix}
            I_{d} \\ 0_{d}
        \end{pmatrix} \equiv \begin{pmatrix}
            I_{d} \\ 0_{d}
        \end{pmatrix}, \ 
        N_N = M_{N - 1}N_1 = 0_d, \ 
        O_N = M_{N - 1}O_1 = 0_d \Mod{\ell^j}.
\end{gather*}
Therefore, when $i = N$, the conclusion of Lemma \ref{E8} yields 
\begin{align} \label{selfcong}
    \vec{f}\mid U_{m^{2N}} = \vec{f} \mid U_{m^{2}}^{N} & = M_{N} \vec{f} + N_{N} (\vec{f} \otimes 1_{m}) + O_{N} (\vec{f} \mid V_{m^{2}})\equiv \vec{f} \Mod{\ell^{j}}.
\end{align}
We use \eqref{lin_comb} and \eqref{selfcong} to compute
\begin{align*}
   \sum p_{[1, p]}\left( \frac{\ell^{j} n + 1}{D}    \right) q^{n} & \equiv \sum_{i = 1}^{d} \beta_{i} f_{i} \equiv \sum_{i = 1}^{d}\beta_i f_i\mid U_{m^{2N}}
   \equiv \sum p_{[1, p]}\left( \frac{\ell^{j} n + 1}{D}\right) q^{n}\mid U_{m^{2N}} \\
   & \equiv \sum p_{[1, p]}\left( \frac{\ell^{j} m^{2N}n + 1}{D}\right) q^{n} \Mod{\ell^j}.
\end{align*}
Comparing coefficients, letting $w$ be a non-negative integer, and replacing $n$ by $m^wn$ in the resulting congruence gives~\eqref{E14}.


\subsection{Explicit examples of congruences}
We conclude with details on the congruences \eqref{ex1} and \eqref{ex2} illustrating Theorem \ref{E5}.  To explain \eqref{ex1}, we let $\ell =13$, $m = 7$, $p = 5$, and $j = 1$.  When $p = 5$, we have $D = 4$.  With $f_0$, $f_1$, and $f_2$ as in \eqref{basis}, Theorem \ref{MT1} and the example illustrating the theorem imply that 
\[
\sum_{n = 0}^{\infty}p_{[1,5]}\left(\frac{13n+1}{4}\right)q^n \equiv
(\eta(4z)\eta(20z))^3H(4z) \in \mathcal{A}_{5,3,8,1_5}\Mod {13}
\]
with $H(z) = 12f_0(z)^2 + 2f_0(z)f_1(z) + 6f_1(z)^2+3f_1(z)f_2(z) + f_2(z)^2 \in M_8(\Gamma_0(5))$.  We let $g(z) = (\eta(4z)\eta(20z))^3 \in M_{3}\left(\Gamma_0(80), \left(\frac{\bigcdot}{5}\right)\right)$.  Since $\mathcal{A}_{5, 3, 8, 1_5}\cong M_8(\Gamma_0(5))$ is $5$-dimensional, $\mathcal{A}_{5, 3, 8, 1_5}$ has $\Z$-basis
\[
B = \{g(z)f_0(4 z)^{2 - a}f_1(4 z)^{a} : 0\leq a\leq 2\} \cup \{g(z)f_1(4 z)^{2 - b}f_2(4 z)^b : 1\leq b\leq 2\}.
\]
By Theorem \ref{warnock_boylan}, the space $\mathcal{A}_{5, 3, 8, 1_{5}}$ is invariant under the Hecke operators $T_{m^2}$ in the ambient space $M_{11}\left(\Gamma_0(80), \left(\frac{\bigcdot}{5}\right)\right)$.  We let $M$ be the matrix of $T_{49}$ with respect to the basis $B$.  Using PARI/GP \cite{PARI2}, we compute   
\begin{equation*}
   M \equiv \begin{pmatrix}
       10 & 3 & 6 & 3 & 11 \\
       9 & 12 & 11 & 12 & 5 \\
       11 & 9 & 1 & 6 & 2 \\
       12 & 1 & 10 & 12 & 6 \\
       11 & 2 & 7 & 11 & 10 
   \end{pmatrix} \Mod{13}.
\end{equation*}
We find that the order of the matrix 
\begin{equation*}
   A = \begin{pmatrix} M - \left( \left(   \frac{-5}{7}  \right)7^{10} \right) I_{5} & (- 7^{20})  I_{5} \\ I_{5} & 0_{5} \end{pmatrix}
\end{equation*}
as in \eqref{A_matrix} in $\textup{PGL}_{10}\left(\mathbb{Z}/13 \mathbb{Z}\right)$ is $J = 1190$, which gives the congruence \eqref{ex1}.  Furthermore, we find that $A$ has order
$3570$ in $\textup{GL}_{10}\left(\mathbb{Z}/13 \mathbb{Z}\right)$.  The second part of Theorem \ref{E5} then implies, for all non-negative integers $w$ and $n$, the congruence
		\begin{equation*}
			p_{\left[ 1, 5  \right]} \left( \frac{13 \bigcdot 7^{w} \bigcdot n + 1}{4}  \right) \equiv  p_{\left[ 1, 5  \right]} \left( \frac{13 \bigcdot 7^{7140 + w} \bigcdot n + 1}{4} \right)  \Mod{13}.  
		\end{equation*} 

\vskip.1in

We now turn to the congruence \eqref{ex2}.  We let $\ell = 7$, $m = 23$, $p = 3$, and $j = 2$.  When $p = 3$, we have $D = 6$.  Theorem \ref{MT1} implies that there exists $H(z)\in M_{36}(\Gamma_0(3))$ such that 
\[
\sum_{n = 0}^{\infty}p_{[1,3]}\left(\frac{7^2n + 1}{6}\right)q^n \equiv (\eta(6z)\eta(18z))^5 H(6z)\in \mathcal{A}_{3, 5, 36, 1_3}\Mod {7^{2}}. 
\]
We let $h(z) = (\eta(6z)\eta(18z))^{5}\in M_{5}\left(\Gamma_0(108), \left(\frac{-3}{\bigcdot}\right)\right)$, and we let 
\[
g_{0}(z) = \frac{\eta(z)^{18}}{\eta{(3 z)}^{6}} = 1 + \cdots,\, g_{1}(z) = (\eta(z) \eta(3 z))^{6} = q + \cdots, \, g_{2}(z) = \frac{\eta(3 z)^{18}}{\eta{(z)}^{6}} = q^2 + \cdots
\]
in $M_{6}(\Gamma_0(3))$.  Since $\mathcal{A}_{3, 5, 36, 1_3} \cong M_{36}(\Gamma_0(3))$ is $13$-dimensional, $\mathcal{A}_{3,5,36,1_3}$ has $\Z$-basis
\[
B'= \{h(z)g_0(6 z)^{6-a}g_1(6 z)^a : 0\leq a\leq 6\}
\cup \{h(z)g_1(6 z)^{6-b}g_2(6 z)^b: 1\leq b\leq 6\}.  
\]
Theorem \ref{warnock_boylan} implies that $\mathcal{A}_{3, 5, 36, 1_3}$ is $T_{m^2}$-invariant in the ambient space $M_{41}\left(\Gamma_0(108), \left(\frac{-3}{\bigcdot}\right)\right)$.  We let $M'$ be the matrix of $T_{23^2}$ with respect to $B'$.  Using PARI/GP \cite{PARI2}, we compute

\[
M' \equiv \begin{pmatrix}
          1 & 3 & 0 & 10 & 11 & 44 & 18 & 10 & 28 & 45 & 16 & 26 & 34 \\
          15 & 22 & 9 & 9 & 36 & 38 & 5 & 40 & 24 & 42 & 36 & 1 & 39 \\
          14 & 2 & 38 & 10 & 7 & 6 & 41 & 47 & 7 & 15 & 8 & 12 & 31 \\
          36 & 15 & 31 & 20 & 4 & 33 & 34 & 41 & 0 & 40 & 3 & 14 & 40 \\
          33 & 30 & 22 & 21 & 28 & 14 & 38 & 44 & 2 & 29 & 30 & 6 & 39 \\
          47 & 2 & 19 & 2 & 34 & 16 & 10 & 24 & 10 & 15 & 7 & 1 & 41 \\
          3 & 4 & 15 & 2 & 37 & 7 & 3 & 7 & 16 & 44 & 43 & 46 & 31 \\
          41 & 8 & 7 & 36 & 45 & 31 & 38 & 16 & 27 & 30 & 26 & 9 & 26 \\
          11 & 6 & 44 & 43 & 44 & 2 & 45 & 14 & 28 & 21 & 36 & 23 & 26 \\
          19 & 14 & 3 & 26 & 0 & 20 & 6 & 12 & 25 & 20 & 3 & 29 & 8 \\
          17 & 40 & 1 & 15 & 7 & 19 & 20 & 27 & 7 & 38 & 38 & 16 & 14 \\
          25 & 29 & 22 & 42 & 24 & 26 & 26 & 3 & 15 & 30 & 44 & 22 & 22 \\
          20 & 33 & 23 & 3 & 28 & 10 & 46 & 23 & 46 & 45 & 0 & 31 & 1 
      \end{pmatrix}\Mod{7^{2}}. 
\]
We find that the order of the matrix 
\begin{gather}
   A' = \begin{pmatrix} M - \left(  \left( \frac{-3}{23} \right) 23^{40} \right) I_{13} & (- 23^{80})  I_{13} \\ I_{13} & 0_{13} \end{pmatrix} \Mod{7^{2}}
\end{gather}
in $\textup{PGL}_{26}\left(\mathbb{Z}/49\mathbb{Z}\right)$ is $J = 1176$.  The congruence 
\eqref{ex2} follows.  Additionally, we find that $A'$ has order $1176$ in $\textup{GL}_{26}\left(\mathbb{Z}/49 \mathbb{Z}\right)$.  An application of the second part of Theorem \ref{E5} now yields the congruence
		\begin{gather}
			p_{\left[ 1, 3  \right]} \left( \frac{7^{2} \bigcdot 23^{w} \bigcdot n + 1}{6}  \right) \equiv  p_{\left[ 1, 3  \right]} \left( \frac{7^{2} \bigcdot 23^{2352 + w} \bigcdot n + 1}{6} \right)  \Mod{7^{2}}
		\end{gather} 
for all non-negative integers $w$ and $n$.
\section{Data availability}
Data sharing is not applicable to this article as no datasets were generated or analysed during the current study.

\end{document}